\documentclass[final,leqno,showlabe]{siamltex}
\usepackage{amsmath}
\usepackage{amssymb}
\usepackage{cases}
\usepackage{enumerate}

\input psfig.sty

\usepackage{mathrsfs}
\usepackage{amsfonts}
\usepackage{multirow}

\usepackage{tabularx}
\usepackage{hyperref}

\usepackage{epsfig,epstopdf,color,bm}
\usepackage{color}

\newtheorem{exmp}{\bf Example}[section]
\newtheorem{remark}{\bf Remark}[section]
\newtheorem{lem}{\bf Lemma}[section]
\newtheorem{theor}{\bf Theorem}[section]

\usepackage{graphicx}
\graphicspath{{./figs/}}
\usepackage{float}
\usepackage{epstopdf}

\newcommand{\bx}{\bm{x}}

\newcommand{\BH}{{\bf H}}

\newcommand{\fl}[2]{\frac{#1}{#2}}
\newcommand{\p}{\partial}

\newcommand{\bea}{\begin{eqnarray}}
\newcommand{\eea}{\end{eqnarray}}
\newcommand{\beas}{\begin{eqnarray*}}
\newcommand{\eeas}{\end{eqnarray*}}
\newcommand{\be}{\begin{equation}}
\newcommand{\ee}{\end{equation}}
\newcommand{\bes}{\begin{equation*}}
\newcommand{\ees}{\end{equation*}}

\title{A normalized gradient flow method for computing ground states of spin-2 Bose-Einstein condensates\thanks{{\bf Funding}: The work was partially supported by
the National Key R\&D Program of  China (No. 2024YFA1012803) (Q. Tang),  the National Natural Science Foundation of China  Grant No. 11971335 (Q. Tang) and 11971007 (Y. Yuan)
and the Ministry of Education of Singapore
 under its Academic Research Fund MOE-T2EP20122-0002 (A-8000962-00-00) (W. Bao).}}
\author{Weizhu Bao\thanks{Department of Mathematics, National
University of Singapore, Singapore 119076 ({\it
matbaowz@nus.edu.sg}, URL: https://blog.nus.edu.sg/matbwz/).}
\and Qinglin Tang\thanks{
School of Mathematics, Sichuan University, Chengdu 610064, P. R. China
({\it qinglin\_tang@scu.edu.cn},
URL: http://www.qinglin-tang-scu.com).}
\and Yongjun Yuan\thanks{LCSM (MOE),
School of Mathematics and Statistics,
Hunan Normal University, Changsha, Hunan 410081, P. R. China
({\it yyj1983@hunnu.edu.cn}).}}

\date{\today}

\begin{document}
\maketitle

\begin{abstract} We propose and analyze an efficient and accurate numerical method for computing ground 
states of spin-2 Bose-Einstein condensates (BECs) by using the normalized gradient flow (NGF). 
In order to successfully extend the NGF to spin-2 BECs which has
five components in the vector wave function but with only two physical constraints on total mass 
conservation and magnetization conservation, 
two important techniques are introduced for designing the proposed numerical method. The first one is 
to systematically investigate the ground state structure and property of spin-2 BECs within  
a spatially uniform system, which can be used on how to 
properly choose initial data in  the NGF for computing ground 
states of spin-2 BECs. The second one is 
to introduce three additional projection conditions based 
on the relations between the chemical potentials,  together with the two existing 
physical constraints, such that the five projection parameters used in 
the projection step of the NGF can be uniquely  determined. 
Then a backward-forward Euler finite difference method is adapted to discretize 
the NGF. We prove rigorously that there exists a unique solution of 
the nonlinear system for determining the five projection parameters in the full discretization
of the NGF under a mild condition on the time step size. 
Extensive numerical results on ground states of spin-2 BECs {\color{black}with different types of phases and under different potentials are reported to show} 
the efficiency and accuracy of the proposed numerical method and to demonstrate 
several interesting physical phenomena on ground states of spin-2 BECs.  
\end{abstract}



\pagestyle{myheadings} \markboth{Weizhu Bao, Qinglin Tang and Yongjun Yuan}
{Ground states of spin-2 BECs}

\section {Introduction} 
\setcounter{equation}{0}

Since its experimental realization 
 in 1995 \cite{AEMWC1995, BSTH1995, DMADDKK1995}, 
 the  Bose-Einstein condensate (BEC)
has stimulated great excitement in the physical community and regains vast interests 
in atomic and molecular as well as condense matter physics.  
At early stage, atoms were magnetically trapped in BEC experiments and hence their
spin degrees of freedom were frozen \cite{AEMWC1995, BSTH1995, DMADDKK1995}.  
Nevertheless,  recently developed optical trapping techniques \cite{SA} have enabled  to release 
the spin internal degrees of freedom,  opening up a new research arena of quantum 
many-body systems named spinor BEC \cite{CH,  GG, Pasq2011}.   
Extensive theoretical and experimental studies have been carried out to reveal numerous 
new quantum  phenomena which are generally absent in a spin-frozen condensate 
\cite{BM, CY, IO,  KU2, KU1,  KA, WG,XX,YJ}. Within the mean-field approximation, 
in contrast with a spin-frozen BEC whose order parameter can be well described  by a 
{\color{black} scalar wave function $\phi$, a spin-$F$  $(F\in \mathbb{N})$  BEC is described by a
macroscopic complex-valued vector wave function $\Phi$  consisting of  $2F+1$ components, 
each of which characterizes  one of the $2F+1$ hyperfine states  $(m_F=-F, \cdots, F)$.}
 In this paper, we consider spin-2 BECs, i.e. $F=2$. 
 
{\color{black}  One important problem in the theoretical study of a spin-2 BEC is 
to find its ground state  so as  to initialize its dynamics and to  predict new 
important phases of the ground state  which can be later compared with or confirmed  
by  those physical  experimental observations.  The ground state  of a spin-2 BEC   is 
defined as the minimizer of the following dimensionless Gross-Pitaevskii (GP) energy functional 
\cite{BCY,KU2,W2}:
\begin{equation}
\label{energy}
 \mathcal{E}(\Phi)=\int_{\mathbb{R}^d}\bigg[\ \sum_{\ell=-2}^2\left(\frac{1}{2}|\nabla\phi_\ell|^2+V(\bx)|\phi_\ell|^2\right)
    +\frac{\beta_0}{2}\rho^2+\frac{\beta_1}{2}|\mathbf{F}|^2 +\frac{\beta_2}{2}|A_{00}|^2 \bigg]\;d\bx,
\end{equation} 
satisfying both the {\em mass \eqref{mass}} and {\em magnetization} \eqref{magnet} constraints:
\bea
\label{mass}
 \mathcal{N}(\Phi(\bx))&:=&\sum_{\ell=-2}^2 \int_{\mathbb{R}^d}|\phi_\ell(\bx)|^2d\bx=1,\\
 \label{magnet}
  \mathcal{M}(\Phi(\bx))&:=&\sum_{\ell=-2}^2\int_{\mathbb{R}^d}\ell |\phi_\ell(\bx)|^2d\bx=M.
\eea
Here, $M\in [-2, 2]$ is a given constant, 
$\bx\in \mathbb{R}^d$ ($d=1,2,3$) is the spatial variable, 
$\Phi(\bx)=:\big(\phi_{2}(\bx), \phi_{1}(\bx),\phi_{0}(\bx),\phi_{-1}(\bx),\phi_{-2}(\bx)\big)^T$ is the  wave function,
$\rho:=\sum_{\ell=-2}^2\rho_\ell$ is the total density 
with $\rho_\ell=|\phi_\ell(\bx)|^2$ being the $\ell$-th component density,  $\beta_0$, $\beta_1$, 
$\beta_2$ are real constants characterizing the spin-independent interaction, 
spin-exchange interaction and spin-singlet interaction, respectively. 
In addtion,  $V(\bx)$ is a real-valued function represents the external trapping potential,
$A_{00}(\Phi):=\big(2\phi_{2}\phi_{-2}-2\phi_{1}\phi_{-1}+\phi_0^2\big)/\sqrt{5}$ 
is the amplitude of the spin-singlet pair and
$\mathbf{F}(\Phi):=[F_x(\Phi),F_y(\Phi),F_z(\Phi)]^{\top}$ is the spin vector 
with its components defined as  \cite{BCY}: 
\begin{equation*}
\begin{split}
&F_x={\rm Re} ( F_{+} ), \  F_y={\rm Im} ( F_{+} ),
\  \mbox{with} \  F_+=2\big(\bar{\phi}_{2}\phi_{1}+\bar{\phi}_{-1}\phi_{-2}\big)+\sqrt{6}\big(\bar{\phi}_{1}\phi_{0}+\bar{\phi}_{0}\phi_{-1}\big),\\	 
&F_z=2\big(|\phi_{2}|^2-|\phi_{-2}|^2 \big)+|\phi_{1}|^2-|\phi_{-1}|^2,
\end{split}
\end{equation*}
where $\bar{f}$ denotes the complex conjugate of $f$. Therefore, the ground state  $\Phi_g(\bx)$ is the
solution of the following nonconvex minimization problem:
\be
\label{ground}
\Phi_g :=\arg \min_{\Phi\in \mathbb{S}} \mathcal{E}(\Phi), 
\ee
where the nonconvex set $\mathbb{S}$ is defined as}
\begin{equation}
\label{ground-const}
 \mathbb{S}=\left \{\Phi=(\phi_2,\phi_1,\phi_0,\phi_{-1},\phi_{-2})^T\in\mathbb{C}^5\ | \
 \mathcal{N}(\Phi)=1,\ \mathcal{M}(\Phi)=M,\ \mathcal{E}(\Phi)<\infty \right\}.
\end{equation} 
 
  Existence and uniqueness of the ground state  \eqref{ground} was carried out in \cite{BCY,KU2,KU1}.
  Meanwhile, validity of the so called single mode approximation (SMA) of the ground state 
  (which  simplifies the ground state computation) is
partially investigated.  Different numerical methods have been proposed to compute the ground state of a scalar BEC
 \cite{AD,AD2014,ALT2017,B1,BC,BCL,BD,BT,BW,CC1,CC2,CL,CS,DK2010,WWB} and a spin-1 BEC 
 \cite{BCZ,BL,BW,CL2021,LC2021,YXTW2018,ZY,ZYY}. 
 Among  them, a simple and most popular method is the normalized gradient flow (NGF) (or imaginary time method)
 incorporated with a proper discretization  scheme to evolve the resulted gradient flow under 
 normalization of the wave function \cite{B1,BD,BL,BW}.
However, to extend the NGF from single-component BEC and spin-1 BEC to spin-2 BEC,  
due to the fact that we only have two conservation conditions \eqref{mass} and \eqref{magnet}
 and there are five projection 
constants to be determined in the projection step, it is unclear that
the NGF method could be easily and straightforwardly extended to compute 
ground states of spin-2 BECs.
A projection gradient method \cite{W2} was proposed to compute ground states of spin-2 BECs, 
where a continuous normalized gradient flow (CNGF) was discretized 
by the Crank-Nicolson finite difference method with a proper and  special way
 to deal with the nonlinear terms. This scheme was proved to be energy-diminishing and 
 conserve both the total mass and magnetization in the discrete level.
However,  a fully nonlinear coupled system need to be solved  at each time step,
which introduces much computational cost, especially in three dimensions.  
Recently, numerical methods were presented for computing ground states of spin-2 BEC
based on the NGF with inaccurate projections \cite{CL2021,TCW}.  
 The main  objective of this paper is to present and analyze a numerical method
 for computing ground states of spin-2 BECs via the NGF.
 In order to do so, two main techniques are presented, which are:  
(i)  to carry out a systematic study on the ground state structure and property 
of a spin-2 BEC within a spatially uniform system, which is then used to choose 
simple and proper initial data in the NGF for computing ground states of
spin-2 BECs,   and 
(ii) to introduce three additional projection  conditions based on the relations between 
the chemical potentials of a spin-2 BEC for overcoming the fact that there are five components in the vector wave function but with only two constraints on total
mass conservation and total magnetization conservation. 
In fact, the proposed three additional projection  conditions, together with  
total mass conservation and magnetization conservation, can completely  
determine  the five projection constants used in projection step of the NGF.  
This enables us to extend the  simple and powerful  NGF method
to compute ground states of spin-2 BECs.

The rest of the paper is organized as follows. In section 2, 
 ground state structure and property of spin-2 BEC in a spatially 
 uniform system are investigated systematically.
In section 3, a NGF is constructed by introducing three additional 
 projection conditions and then  a backward-forward Euler finite difference method 
is presented to discretize the NGF. 
 In section 4,  we report extensive numerical results on ground states of 
 spin-2 BECs with {\color{black} different types of phases and under 
different potentials in  one and two dimensions.}
 Finally,  some conclusions are drawn in section 5.

\section{Ground states and their properties}
In this section, we mainly investigate the ground state structure and property of a
spin-2 BEC.

\subsection{Euler-Lagrange equations and classification of ground states}
\label{sec:class_phase}
The Euler-Lagrange equation associated to the  minimization  problem  \eqref{ground}  reads as
\be\label{eq:Euler_Lagran_eq1}
\left\{
\begin{split}
(\mu \pm 2\lambda)\phi_{\pm 2}
&=\left(\textcolor{black}{H_\rho} \pm2 \beta_1F_z\right)\phi_{\pm 2}+
\beta_1F_{\mp}\phi_{\pm1}+\frac{\beta_2}{\sqrt{5}}A_{00}\bar{\phi}_{\mp 2}, 	\\[-0.3cm] 
(\mu \pm \lambda)\phi_{\pm 1}
&=\left(\textcolor{black}{H_\rho} \pm \beta_1 F_z\right)\phi_{\pm 1}
	+\beta_1\left(\frac{\sqrt{6}}{2}F_{\mp}\phi_{0}+F_{\pm}\phi_{\pm 2}\right)
	-\frac{\beta_2}{\sqrt{5}}A_{00}\bar{\phi}_{\mp 1},\\[-0.3cm] 
\mu \phi_{0}
&=\textcolor{black}{H_\rho} \phi_{0}+\frac{\sqrt{6}}{2}\beta_1\big(F_{+}\phi_{1}+F_{-}\phi_{-1}\big)
	+\frac{\beta_2}{\sqrt{5}}A_{00}\bar{\phi}_{0}.
\end{split}
\right.
\ee
{\color{black}
Here, $H_\rho=-\nabla^2/2+V(\bx)+\beta_0\rho$, $F_{-}=\bar{F}_{+}$, 
$\mu$ and  $\lambda$  are  Lagrange multipliers associated to the mass and magnetization  
constraints \eqref{mass}-\eqref{magnet}.}
Thus the ground state of \eqref{ground} can also be viewed as the  eigenfunction of the
 nonlinear  eigenvalue 
problem \eqref{eq:Euler_Lagran_eq1} with constraints
\eqref{mass}-\eqref{magnet}, which has the lowest energy among all eigenfunctions. 
Other eigenfunctions with higher energy are  called as excited states.

\textcolor{black}{As carried out in  \cite{BCY}, the ground state of \eqref{ground} 
is unique up to a phase-rotation, i.e.,}
$\Phi_g^1:=(\phi_2^{1,g},\phi_1^{1,g},\phi_0^{1,g},\phi_{-1}^{1,g},\phi_{-2}^{1,g})^T$ and 
$\Phi_g^2:=(\phi_2^{2,g},\phi_1^{2,g},\phi_0^{2,g},\phi_{-1}^{2,g},\phi_{-2}^{2,g})^T$ are regarded as 
the same if there exists a \textcolor{black}{constant vector} 
  $\bm{\alpha}:=(\alpha_2,\alpha_1,\alpha_0, \alpha_{-1},\alpha_{-2})^T
 =\left(e^{i (2\theta_1-\theta_0)},\, e^{i\theta_1},\, e^{i\theta_0}\right.$, $\left. e^{i (2\theta_0-\theta_1)},
 \, e^{i (3\theta_0-2\theta_1)}\right)^T$  with $\theta_0,\theta_1\in\mathbb{R}$ such that  
 $\phi_\ell^{1,g}=\alpha_\ell \phi_\ell^{2,g}$ for $\ell=-2,\cdots,2$. 
According to \cite{CY, KU2}, \textcolor{black}{the phase of} the ground state $\Phi_g$ of \eqref{ground} 
can be classified into three categories based on the values of $|F_+(\Phi_g)|$ and $|A_{00}(\Phi_g)|$: 
(i) ferromagnetic  \textcolor{black}{phase} if $|F_+(\Phi_g)|>0$ and $A_{00}(\Phi_g)=0$, 
(ii) nematic  \textcolor{black}{phase}  if $F_+(\Phi_g)=0$  and $|A_{00}(\Phi_g)|>0$, and 
(iii) cyclic  \textcolor{black}{phase}  if $F_+(\Phi_g)=A_{00}(\Phi_g)=0$. 
{\color{black} When $M=2$ or $-2$, the constraints \eqref{mass}-\eqref{magnet}  
only allow one component, i.e. $\phi_2$ or $\phi_{-2}$, to be nonzero. Therefore,  
 \eqref{ground} can be reduced to compute ground state of a single component BEC with  $\phi_2$ or $\phi_{-2}$, 
which has been well studied  \cite{BCL, BD, BJ, BT}.} 
In addition, if one replaces the magnetization $\mathcal{M}(\Phi)=M$ in \eqref{ground-const}
by $\mathcal{M}(\Phi)=-M$, it is easy to see that the ground state $\Phi_g:=
(\phi_2^g,\phi_1^g,\phi_0^g,\phi_{-1}^g,\phi_{-2}^g)^T$ of \eqref{ground} 
can be simply replaced by $\tilde{\Phi}_g:=
(\phi_{-2}^g,\phi_{-1}^g,\phi_0^g,\phi_{1}^g,\phi_{2}^g)^T$. 
Thus for the simplicity of notations and presentation,  
from now on, we only consider the magnetization $M\in[0,2)$.   

\medskip

\begin{remark}
\label{RM1.1} 
In the literature \cite{BCY,KU2}, instead of the ground state defined as the minimizer
of the energy function $\mathcal{E}(\Phi)$ under two constraints of the total 
mass conservation $\mathcal{N}(\Phi)=1$ and total magnetization $\mathcal{M}(\Phi)=M$ with 
$M\in[-2,2]$, i.e. \eqref{ground}, another type of ground state has also been 
studied. {\color{black}It} is defined as the minimizer
of  energy function $\mathcal{E}(\Phi)$ under only  the mass constraint \eqref{mass}, i.e. 
\be
\label{groundnorm}
\tilde{\Phi}_g :=\arg \min_{\|\Phi\|=1} \mathcal{E}(\Phi).
\ee
In fact, the above minimization problem can be obtained via the minimization problem \eqref{ground}
by further minimizing for $M\in[-2,2]$, i.e. 
\be
\label{groundnorm11}
\tilde{\Phi}_g :=\arg \min_{M\in[-2,2]}\Phi_g^M, \quad \hbox{\rm with}\quad 
\Phi_g^M :=\arg \min_{\Phi\in \mathbb{S}} \mathcal{E}(\Phi).
\ee
The Euler-Lagrange equation associated to the minimization  problem  \eqref{groundnorm} 
is given by
\be
\label{eq:Euler_Lagran_eq135}  
\left\{
\begin{split}
\mu\phi_{\pm 2}&=\left(\textcolor{black}{H_\rho} +\beta_0\rho \pm2 \beta_1F_z\right)\phi_{\pm 2}+
\beta_1F_{\mp}\phi_{\pm1}+\frac{\beta_2}{\sqrt{5}}A_{00}\bar{\phi}_{\mp 2}, 	\\[-0.3cm]
\mu \phi_{\pm 1}
&=\left(\textcolor{black}{H_\rho} +\beta_0\rho \pm \beta_1 F_z\right)\phi_{\pm 1}+\beta_1\left(\frac{\sqrt{6}}{2}F_{\mp}\phi_{0}+F_{\pm}\phi_{\pm 2}\right)
	-\frac{\beta_2}{\sqrt{5}}A_{00}\bar{\phi}_{\mp 1},\\[-0.3cm]
\mu \phi_{0}
&=\left (\textcolor{black}{H_\rho} +\beta_0\rho \right) \phi_{0}+\frac{\sqrt{6}}{2}\beta_1\big(F_{+}\phi_{1}+F_{-}\phi_{-1}\big)
	+\frac{\beta_2}{\sqrt{5}}A_{00}\bar{\phi}_{0},
\end{split}
\right.
\ee
\textcolor{black}{where $\mu$ is the Lagrange multiplier associated to the mass constraint \eqref{mass}.}
\end{remark}

\subsection{Ground states in a spatially uniform system}
In this setion, we consider a spin-2 BEC in a spatially  uniform system, i.e. the \textcolor{black}{GP functional \eqref{energy}} 
without potential (i.e.,  $V(\bx)\equiv 0$) on a bounded domain $\mathcal{D}$ with measure  
$|\mathcal{D}|=1$ and periodic boundary condition. 
\textcolor{black}{This will be applied directly in Section \ref{SMA}   to construct the so-called  
single mode approximation (SMA) of  the ground states in a spatial non-uniform system, which
will reduce significantly the difficulty and computational cost to obtain a ground state of spin-2 BECs.
In addition, in the parameter region where SMA is invalid, the ground states carried out here can 
 help build more efficient initial data for the algorithm proposed in Section  \ref{sec:algorithm}
 to accelerate its convergence.  }
%

 In  a spatially  uniform system, all  ground states are constants in  form of  \cite{KU2}
\be
\label{Uniform-GS}
\Phi(\bx)\equiv \bm{\xi}:=(\xi_2,\xi_1,\xi_0,\xi_{-1},\xi_{-2})^T,\quad \bx\in \overline{\mathcal{D}},
\ee 
where $\xi_j\in\mathbb{C}$  for $j=-2, -1, 0, 1, 2$. Plugging \eqref{Uniform-GS}
into \eqref{energy} with $V(\bx)\equiv 0$ and $\Phi(\bx)=\bm{\xi}$ and replacing 
$\mathbb{R}^d$ by $\mathcal{D}$, after a detailed computation, we obtain
\be
\label{energy_uniform}
\mathcal{E}(\Phi)=\mathcal{E}_U(\bm{\xi}):=\fl{1}{2} \left[\left(\beta_1 |\tau|^2+\fl{\beta_2}{5}|\delta|^2 \right) 
+\bigg(\beta_0 +\beta_1 M^2\bigg) \right]=: E\left(\tau,\delta \right),
\ee
where
\be
\label{formula_tau_delta}
\left\{
\begin{array}{l}
\tau:=\tau(\bm{\xi}) =F_{+}(\bm{\xi})=2\big(\bar{\xi}_{2}\xi_{1}+\bar{\xi}_{-1}\xi_{-2}\big)
     +\sqrt{6}\big(\bar{\xi}_{1}\xi_{0}+\bar{\xi}_{0}\xi_{-1}\big), \\[0.5em]
 \delta :=\delta(\bm{\xi})=A_{00}(\bm{\xi})=2\xi_2\xi_{-2}-2\xi_1\xi_{-1}+\xi_0^2.
\end{array}
\right.
\ee
{\color{black}Actually, $\tau$ and $\delta$ are precisely the quantities which will be used later to characterize ground states as either, ferromagnetic, nematic or cyclic.} 
The conservation of  {\em mass} \eqref{mass} and {\em magnetization }
 \eqref{magnet}  leads to:   
\be 
\label{conservation_uniform}
\left\{
\begin{array}{lr}
|\xi_2|^2+|\xi_{-2}|^2+|\xi_0|^2+|\xi_{1}|^2+|\xi_{-1}|^2=1, & \\[0.25em]
2(|\xi_2|^2-|\xi_{-2}|^2)+|\xi_{1}|^2-|\xi_{-1}|^2=M,  & 
\end{array}
\right.
\ee 
{\color{black}for given $M\in[0, 2)$.} Solving   \eqref{ground}-\eqref{ground-const} for  ground state
$\Phi_g(\bx)$
is now equivalent to solve  the following minimization problem for  minimizer
$\bm{\xi}_g:= (\xi^g_2,\xi^g_1,\xi^g_0,\xi^g_{-1},\xi^g_{-2})^T$ as
\be
\label{ground-uniform}
\quad \bm{\xi}_g :=\arg \min_{\bm{\xi}\in \mathbb{S}_C} \mathcal{E}_U(\bm{\xi}), 
\quad
 \mathbb{S}_C=\left \{ \bm{\xi}\in\mathbb{C}^5 \ \bigg|  \  \sum_{\ell=-2}^2 |\xi_{\ell}|^2=1, \  \sum_{\ell=-2}^2 \ell  |\xi_{\ell}|^2=M\right\}.
\ee
 
 In the following, we show that the minimization problem \eqref{ground-uniform} on  complex manifold 
$ \mathbb{S}_C$ can be reduced to a minimization problem on the real manifold 
$\mathbb{S}_R=\mathbb{S}_C\cap \mathbb{R}^5 $.

\medskip 

\begin{lem}
\label{lemma:equiv_tau_delta}
If $\bm{\xi}\in\mathbb{R}^5$, then the system \eqref{conservation_uniform}  has a real solution if and only if 
\be\label{sol_cond}
\tau^2(\bm{\xi})+4\,\delta^2(\bm{\xi})\leq 4-M^2.
\ee
\end{lem}

\begin{proof} Firstly, we prove the necessary condition, i.e. we assume $\bm{\xi}\in\mathbb{R}^5$ and 
the system \eqref{conservation_uniform}  has a real solution. 
 By \eqref{formula_tau_delta} and \eqref{conservation_uniform}, we have
\be\label{a1}
(\xi_2-\xi_{-2})^2+(\xi_1+\xi_{-1})^2=1-\delta.
\ee
Noticing  $ |\delta|\leq\sum_{\ell=-2}^{2}\xi_\ell^2=1$ and  denoting 
$p=\sqrt{1-\delta}\geq0$, then there exists a  constant $\theta\in[0,2\pi)$ such that
\be
\label{assume}
\left\{
\begin{array}{l}
\xi_{-2}-\xi_2=p \cos{\theta},   \\[0.5em]
\xi_1+\xi_{-1}=p \sin{\theta},
\end{array}
\right.\quad
\Longleftrightarrow
\quad
\left\{
\begin{array}{l}
\xi_{-2}=p \cos{\theta} + \xi_2,   \\[0.5em]
\xi_{-1}=p \sin{\theta} -\xi_1.
\end{array}\right.
\ee
We then prove \eqref{sol_cond} is valid by considering four different parameter cases. 
To simplify the presentation, in the following, we only state the formula for $\xi_0$, $\xi_1$ and $\xi_2$,
and the expressions of  $\xi_{-1}$ and $\xi_{-2}$ can be obtained directly from \eqref{assume}.

\medskip

\noindent Case (i). If $\delta=1$, then $p=0$. From \eqref{assume}, we obtain $\xi_2=\xi_{-2}$ and $\xi_1=-\xi_{-1}$.  By \eqref{conservation_uniform} and  \eqref{formula_tau_delta},
we have 
\bes
M=\tau=0,    \quad \Longrightarrow   \quad	\tau^2+4\delta^2 \leq 4-M^2. 
\ees

\noindent Case (ii).  If $\delta<1$ (and thus $p> 0$) and $\cos\theta=0$, {\color{black}then $\xi_{-2}=\xi_2$ and  $\xi_{-1}=p -\xi_1$. Plugging them into \eqref{formula_tau_delta} and \eqref{conservation_uniform}, by a simple calculation, we obtain 
\bea
\label{eq:xi0}
&&\xi_0=(\tau-2p\xi_2)/(p\sqrt{6}),\\[0.25em]
\label{eq:xi1}
&&2 |\xi_2|^2+|\xi_0|^2+ (M+p^2)^2/(2p^2)=1+M.
\eea
Thereby combining  \eqref{eq:xi0} and  \eqref{eq:xi1}, we have
\be
\label{AppA_CaseII_3}
16p^2\xi_2^2-4p \tau \xi_2+3p^4-6p^2+3M^2+\tau^2=0.
\ee
}
Note that \eqref{AppA_CaseII_3} has a real solution if and only if  the corresponding discriminant 
\be
\label{AppA_CaseII_4}
\Delta:=-48p^2(4\delta^2+\tau^2+4M^2-4)\geq0,
\ee
which immediately implies  
\be
\label{AppA_CaseII_5}
 \tau^2+4\delta^2\leq4-4M^2\leq4-M^2. 
\ee

\noindent Case (iii).  If $\delta<1$ (and thus $p> 0$) and $\sin\theta=0$.  Similarly,
 we obtain 
 \be
 \label{AppA_CaseIII_1}
  \xi_2=\frac{(2p^2+M)\cos\theta}{4p},\quad
  \xi_1=\frac{\tau\cos\theta}{2p}, \quad 
  \xi_0^2=\frac{4-M^2-4\delta^2-4\tau^2}{8p^2}.
\ee
Thus $\xi_0$ is real  if and only if $4-M^2-4\delta^2-4\tau^2\geq0$, which again implies 
\bes
\tau^2+4\delta^2\leq 4-M^2-3\tau^2\leq4-M^2.
\ees

\noindent Case (iv).  If $\delta< 1$ and $\sin(2\theta)\neq0$.
Plugging \eqref{assume}  into  \eqref{conservation_uniform}, we have 
 \bea
 \label{AppA_CaseIV_1}
&&\qquad \xi_2=-(M+2p^2-p^2\sin^2\theta-2p\sin\theta\xi_1)/(4p\cos\theta),	      \\[0.5em]
\label{AppA_CaseIV_2}
&&\qquad \xi_0=\fl{4-6\sin^2\theta}{\sqrt{6}\sin(2\theta)}\,\xi_1 +\fl{M-2p^2+3p^2\sin^2\theta}{2\sqrt{6}p\cos\theta}
+\fl{\tau}{\sqrt{6}p\sin\theta},\\  [0.5em]
\label{AppA_CaseIV_3}
&&
\qquad A \xi_1^2 + B \xi_1+C=0,
\eea
with
\[
\begin{split}
A=&8 \csc^2(2\theta),\\
B=&\fl{ \csc^2\theta \sec^2\theta}{4p}\Big[5\tau\cos\theta+3\tau\cos(3\theta)-5M\sin \theta-8p^2\sin \theta+3M\sin(3\theta)\Big],\\
C=&\fl{1}{2p^2}\Big[\tau^2\csc^2\theta+(M-2p^2)\tau \csc\theta \sec \theta - \frac{1}{2}\sec^2\theta \big(-2M^2+6\delta p^2-Mp^2+p^4\\
&+3p^2(2\delta+M+p^2)\cos(2\theta) -3p^2\tau \sin(2\theta)\big)
\Big].
\end{split}
\]
Similarly, \eqref{AppA_CaseIV_3}  has a real solution if and only if  the corresponding discriminant 
\be
\label{Delta_AppA_CaseIV_4}
\Delta = -\frac{3\csc^2\theta \sec^2 \theta}{p^2} \;\Big[3(M\sin\theta+\tau\cos\theta)^2+4\delta^2+\tau^2+M^2-4\Big]\geq0,
\ee
which implies 
\bes
 4-M^2-(4\delta^2+\tau^2)\ge3(M\sin\theta+\tau\cos\theta)^2\ge0 \quad \Longrightarrow 
 \quad \tau^2+4\delta^2\leq4-M^2.
\ees

Secondly, we prove the sufficient condition, i.e. we assume \eqref{sol_cond} is valid. 
When $\delta=1$, then $M=\tau=0$, and thus $\xi=(0,0,1,0,0)^T$ is a real solution of 
\eqref{conservation_uniform}; when  $\delta<1$ and  $M=0$, then \eqref{AppA_CaseII_4} is satisfied, 
thus \eqref{eq:xi0}-\eqref{AppA_CaseII_3} is a real solution; and  when $\delta<1$ and $M\ne0$, by choosing 
$\theta=\pi-\arctan(\tau/M)$ such that $M\sin\theta=-\tau\cos\theta$ and \eqref{Delta_AppA_CaseIV_4} is fulfilled,
then \eqref{AppA_CaseIV_1}-\eqref{AppA_CaseIV_3} is a real solution. \hfill
\end{proof}

\bigskip

\begin{lem}
\label{lemma:exist_real_sol} 
For $\forall\, \bm{\xi} \in \mathbb{S}_C$, we have
\be
\label{lemma_inequality2}
|\tau(\bm{\xi})|^2+4\,|\delta(\bm{\xi})|^2\leq 4-M^2,
\ee
{\color{black} and}
there exists a $\bm{\zeta}_R\in \mathbb{S}_R$ such that $\mathcal{E}_{U}(\bm{\zeta}_R)=\mathcal{E}_{U}(\bm{\xi})$. Therefore, the minimization problem \eqref{ground-uniform} has at least a real ground state. 
 
\end{lem}

\medskip

\begin{proof} 
We prove this lemma by considering two different cases.

\noindent Case (i). If $\delta(\bm{\xi})=0$,   noticing that 
\bes
\bm{\xi} \in \mathbb{S}_C \qquad \Longrightarrow \qquad  \bm{\zeta}:=\left(|\xi_2|,|\xi_1|,|\xi_0|,|\xi_{-1}|,|\xi_{-2}|\right)^T \in \mathbb{S}_R,
\ees
 hence the system 
\eqref{conservation_uniform} is fulfilled for $\bm{\zeta}$. By lemma \ref{lemma:equiv_tau_delta}, we have  
\bes
|\tau(\bm{\xi})|^2+ 4|\delta(\bm{\xi})|^2 =|\tau(\bm{\xi})|^2 \le \tau^2(\bm{\zeta}) \le \tau^2(\bm{\zeta})+4\delta^2(\bm{\zeta}) \le 4-M^2.
\ees

\noindent  Case (ii). If $\delta(\bm{\xi})\ne0$,  noticing that
\be\label{proof_inequality2}
\begin{split}
\max_{\bm{\xi}\in \mathbb{S}_C}\left\{|\tau(\bm{\xi})|^2+4|\delta(\bm{\xi})|^2+M^2 \right\}
=&\max_{\bm{\xi}\in \mathbb{S}_C}\left\{|\tau(\bm{\xi})|^2+4|\delta(\bm{\xi})|^2+F_z^2(\bm{\xi}) \right\} \\
\le&\max_{\bm{\zeta}\in \mathbb{S}_1}\left\{|\tau(\bm{\zeta})|^2+4 |\delta(\bm{\zeta})|^2+F_z^2(\bm{\zeta}) \right\}, 
\end{split}
\ee
with $\mathbb{S}_1=\left \{ \bm{\zeta}\in\mathbb{C}^5 \ |  \  \sum_{\ell=-2}^2 |\zeta_{l}|^2=1\right \}$, thus to prove   \eqref{lemma_inequality2}, we only need to show 
\be
\label{proof_inequality3}
\max_{\bm{\zeta}\in \mathbb{S}_1}\left\{|\tau(\bm{\zeta})|^2+4|\delta(\bm{\zeta})|^2+F_z^2(\bm{\zeta}) \right\} \le 4.
\ee
Consider an auxiliary minimization problem 
\be
\label{proof_min}
\bm{ \zeta}_g:=(\zeta_2^g,\zeta_1^g,\zeta^g_0,\zeta_{-1}^g,\zeta_{-2}^g)^T=\arg \min_{\bm{\zeta}\in \mathbb{S}_1} \mathcal{F}(\bm{\zeta}), 
\ee
where  the auxiliary functional $\mathcal{F}(\bm{\zeta})$ is defined as
\be
\label{ENERGY}
\mathcal{F}(\bm{\zeta})=- \big[ |\tau(\bm{\zeta})|^2+4|\delta(\bm{\zeta})|^2+F_z^2(\bm{\zeta}) \big].
\ee
It is clear that 
\be
\label{aux_energy}
\mathcal{F}(\bm{\zeta}_g)=- \max_{\bm{\zeta}\in \mathbb{S}_1} \left\{|\tau(\bm{\zeta})|^2+4|\delta(\bm{\zeta})|^2+F_z^2(\bm{\zeta}) \right\} ,
 \ee
 and $\bm{\zeta}_g$ satisfies the  Euler-Lagrange equation 
$\nabla_{\bar{\bm{\zeta}}}\,\mathcal{F}(\bm{\zeta})=\lambda_{\bm{\zeta}}\, \bm{\zeta}$
with $\lambda_{\bm{\zeta}}\in\mathbb{R}$ being the Lagrange multiplier 
and $\bar{\bm{\zeta}}$ being the complex conjugate of $\bm{\zeta}$, 
i.e. 
$\nabla_{\bar{\bm{\zeta}}}\,\mathcal{F}(\bm{\zeta}_g)=\lambda_{\bm{\zeta}_g}\,\bm{\zeta}_g$.  Hence,  
by denoting $\bm{\eta}_g=(\zeta^g_{-2},-\zeta^g_{-1},\zeta^g_0,-\zeta^g_1,\zeta^g_2)^T$,  we have
\be
\label{proof_eq}
\left\{
\begin{array}{l}
\bar{\bm{\zeta}_g}\cdot\nabla_{\bar{\bm{\zeta}}}\,\mathcal{F}(\bm{\zeta}_g)=\lambda_{\bm{\zeta}_g}, \\[0.5em]
\bm{\eta}_g\cdot\nabla_{\bar{\bm{\zeta}}}\,\mathcal{F}(\bm{\zeta}_g)=
\lambda_{\bm{\zeta}_g}\,\bm{\eta}_g\cdot\bm{\zeta}_g,
\end{array}
\right.
\quad
\Longrightarrow
\;\quad
\left\{
\begin{array}{l}
-2(|\tau|^2+4|\delta|^2+F_z^2)=\lambda_{\bm{\zeta}_g}, \\[0.5em]
(\lambda_{\bm{\zeta}_g}+8)\delta=0,
\end{array}
\right.
\ee
which leads to
\be
\label{proof_con}
\left\{
\begin{array}{l}
\lambda_{\bm{\zeta}_g}=-8, \\[0.5em]
 |\tau(\bm{\zeta}_g)|^2+4|\delta(\bm{\zeta}_g)|^2+F_z^2(\bm{\zeta}_g)=4,
\end{array}
\right.
\quad
\Longrightarrow
\;\quad
 \mathcal{F}(\bm{\zeta}_g)=-4.
\ee
Noticing \eqref{aux_energy}, one gets \eqref{proof_inequality3}. 
This, together with \eqref{proof_inequality2}, concludes the desired inequality 
\eqref{lemma_inequality2}.
{\color{black}Therefore, $\forall\,\bm{\xi}\in\mathbb{S}_C$, let $\tau(\bm{\zeta}_R)=|\tau(\bm{\xi})|$ and $\delta(\bm{\zeta}_R)=|\delta(\bm{\xi})|$, we have 
\[
\tau^2(\bm{\zeta}_R)+4\delta^2(\bm{\zeta}_R)\leq 4-M^2. 
\]
According to lemma \ref{lemma:equiv_tau_delta}, the system \eqref{conservation_uniform} has a real solution $\bm{\zeta}_R\in\mathbb{S}_R$, which satisfies $\mathcal{E}_{U}(\bm{\zeta}_R)=\mathcal{E}_{U}(\bm{\xi})$.}
{\color{black}One immediately obtains}
\be
\label{eq:engequivalent}
\min_{\bm{\xi}\in \mathbb{S}_C} \mathcal{E}_U(\bm{\xi})= \min_{\bm{\zeta}_R \in \mathbb{S}_R}\mathcal{E}_{U}(\bm{\zeta}_R)
= \min_{S_{\tau,\delta}} \; E(\tau,\delta)=\frac{1}{2}(\beta_0+\beta_1M^2)+\frac{1}{2}\; \min_{S_{\tau,\delta}}\;f(\tau,\delta),
\ee
where  
\be\label{ftaudelta}
f(\tau,\delta)=\beta_1 \tau^2+\frac{\beta_2}{5}\delta^2 \quad \mbox{and} \quad
S_{\tau,\delta}:=\left\{\tau\in{\mathbb R},\;
\delta\in{\mathbb R},\;\tau^2+4\delta^2\le 4-M^2\right\}.
\ee
Obviously, for any given $\beta_1$ and $\beta_2$, the {\color{black}quadratic} minimization problem
\be 
\label{mpreduced}
(\tau_g,\delta_g):=\arg\; \min_{S_{\tau,\delta}}\;f(\tau,\delta)
\ee
possesses at least a solution {\color{black}over the elliptic region $S_{\tau,\delta}$}. As a result, the minimization problem \eqref{ground-uniform} has at least a real ground state. \hfill 
\end{proof}

\medskip

Thanks to lemmas \ref{lemma:equiv_tau_delta}-\ref{lemma:exist_real_sol}, noticing
\eqref{energy_uniform}, to find the ground state $\bm{\xi}_g$ of \eqref{ground-uniform} is then reduced 
to find minimizers of the minimization problem \eqref{mpreduced}.
Actually, it can be solved analytically since it is to find minimizers of a quadratic function over an elliptic region.
To illustrate this,  Fig. \ref{fig:phase_specific} shows  contour  plots of the energy $E(\tau,\delta)$ 
for  $\beta_0=M=0$ with different  $\beta_1$ \& $\beta_2$  on the elliptic region {\color{black}$S_{\tau,\delta}$}.
From the figure, one can obtain the minimizers $(\tau_g,\delta_g)$. 
\begin{figure}[htbp!]
\centerline{
\psfig{figure=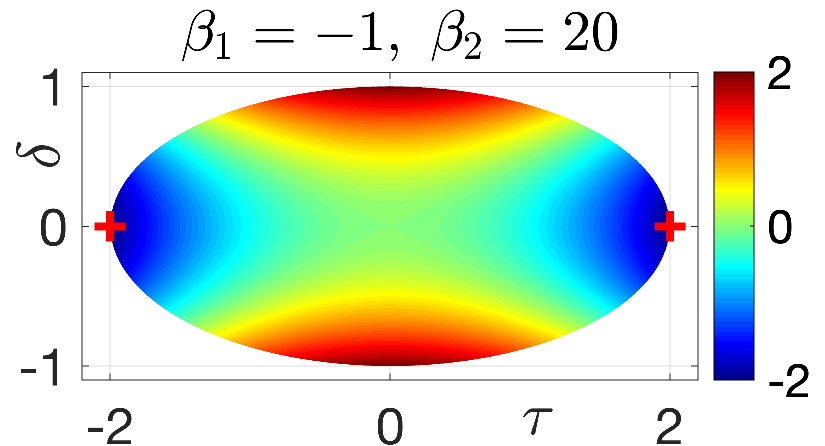,width=1.7in,height=1.3in}
\psfig{figure=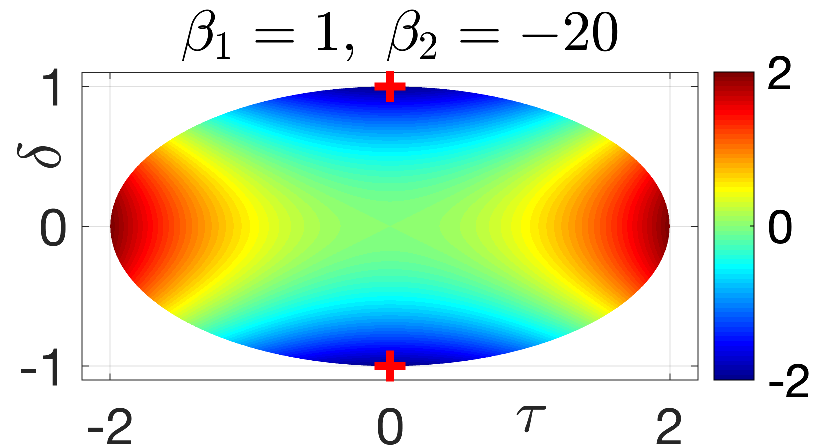,width=1.7in,height=1.3in}
\psfig{figure=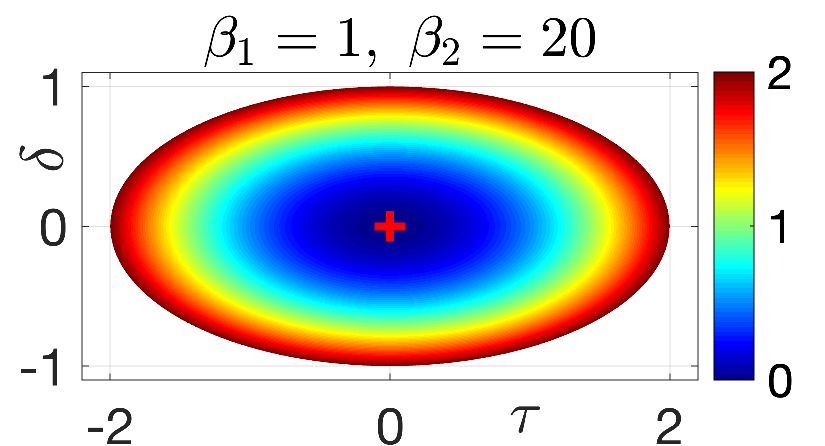,width=1.7in,height=1.3in}
}
\caption{Contour  plots of the energy $E(\tau,\delta)$ in {\color{black}\eqref{eq:engequivalent}} for  $\beta_0=M=0$ with different  $\beta_1=-1$ and $\beta_2=20$ (left), $\beta_1=1$ and $\beta_2=-20$ (middle), 
and $\beta_1=1$ and $\beta_2=20$ (right). $\textcolor{red}{`+'}$ denote those points where the minimum value of $E(\tau,\delta)$  are achieved. According to the value (and position on the graph) of those points $\textcolor{red}{`+'}$, one can immediately conclude the phase of the corresponding ground state
are in ferromagnetic (left), nematic (middle) and cyclic (right).}
\label{fig:phase_specific} 
\end{figure}

Noticing that if $\bm{\xi}$ solves \eqref{formula_tau_delta}-\eqref{conservation_uniform} 
with  $\tau$, then $\tilde{\bm{\xi}}=(\xi_2,-\xi_1,\xi_0,-\xi_{-1},\xi_{-2})^T$ solves
 \eqref{formula_tau_delta}-\eqref{conservation_uniform} with the parameter  $-\tau$. 
Thus, it suffices to assume $\tau \geq 0$ for simplicity hereafter. 
\textcolor{black}{To  solve out analytically the minimization problem \eqref{mpreduced} on the elliptic region $S_{\tau,\delta}$, we adapt the following elliptic-polar coordinates as
\be\label{elliptcoor}
\tau=r \cos\theta,	\quad 
\delta=\fl{1}{2}r \sin\theta,\quad 
\theta\in\left[-\fl{\pi}{2}, \fl{\pi}{2}\right), \quad  
r \in \left[0,\sqrt{4-M^2}\; \right],
\ee
in which coordinate the function $f(\tau,\delta)$ \eqref{ftaudelta} reads as follows: 
\bes
\label{ftaudelta1}
\begin{split}
f(\tau,\delta)&=\beta_1 \tau^2+\frac{\beta_2}{5}\delta^2
=\beta_1 r^2\cos^2\theta+\frac{\beta_2}{20}r^2\sin^2\theta
=\frac{\beta_2}{20}r^2+
\left(\beta_1-\fl{\beta_2}{20}\right)r^2 \cos^2\theta, \\
&=\beta_1r^2+\left(\fl{\beta_2}{20}-\beta_1\right)r^2 \sin^2\theta 
\qquad\quad \theta \in \left[-\fl{\pi}{2}, \fl{\pi}{2}\right), 
\quad  r\in  \left[0,\sqrt{4-M^2}\; \right].
\end{split}
\ees
Therefore,  it is easy to obtain that: 
(i) when  $\beta_1<0$ and $\beta_2>20\beta_1$,  from the last equality, $f(\tau,\delta)$ attains its minimum when  $\sin^2\theta=0$ and $r=\sqrt{4-M^2}$, i.e. at $(\tau_g,\delta_g)=(\sqrt{4-M^2}, 0)$;
(ii) when $\beta_2<0$ and $\beta_2<20\beta_1$, from the third equality, $f(\tau,\delta)$ attains its minimum when  $\cos^2\theta=0$ and $r=\sqrt{4-M^2}$, i.e. at $(\tau_g,\delta_g)=(0,\pm \sqrt{4-M^2}/2)$;
(iii) when $\beta_1>0$ and $\beta_2>0$, from the second equality, $f(\tau,\delta)$ attains its minimum when  $r=0$, i.e., at $(\tau_g,\delta_g)=(0,0)$. 
From \eqref{eq:engequivalent}, the total energy $E(\tau, \delta)$ attains the minimum at the same 
$(\tau_g,\delta_g)$.
Then, from the proof of lemma \ref{lemma:equiv_tau_delta}, one can obtain all real ground states $\bm{\xi}_g\in\mathbb{R}^5$ of \eqref{ground-uniform} by direct calculations
 based on the four different parameter cases
 as shown in the proof of lemma \ref{lemma:equiv_tau_delta}.  
In addition, from  \eqref{formula_tau_delta} $\tau$ and $\delta$ are 
 respectively the value of $F_{+}$ and $A_{00}$, hence according 
 to the standard of classification in {\bf Section \ref{sec:class_phase}},
 the phases of the ground states with parameters $\beta_1, \beta_2$ classified
  in (i), (ii) and (iii) are respectively ferromagnetic, nematic and cyclic. }

\textcolor{black}{Above all,  the real ground state $\bm{\xi}_g\in\mathbb{R}^5$  in a spatially uniform system  can  be thoroughly solved out and  its phase category can also be identified. We summarize  
 the results in {\bf Lemmas \ref{lem:gs_ferr}-\ref{lem:gs_cyc}}, the proofs follow directly  the arguments stated  above.}

\medskip

\begin{lem}
\label{lem:gs_ferr} 
When $\beta_1<0$ and  $\beta_2>20\beta_1$,  
\textcolor{black}{the ground state $\bm{\xi}_g$ is in ferromagnetic phase,
 and for $\forall\,M\in[0, 2)$, it can be solved out  as: }
\be
\label{GSCase2Type1}
\bm{\xi}_g=\left (m_1^4/16,\; m_1^3m_2/8,\; \sqrt{6}m_1^2m_2^2/16,\; m_1m_2^3/8,\; m_2^4/16\right)^T,
\ee
with  $m_1=\sqrt{2+M}, \  m_2=\sqrt{2-M}$.
\end{lem}

\medskip

\begin{lem}
\label{lem:gs_nem} 
When $\beta_2<0$ and $\beta_2<20\beta_1$,  
\textcolor{black}{the ground  state $\bm{\xi}_g$ is in nematic phase.} Moreover, 
if $0<M<2$, then $\bm{\xi}_g$ can be solved out as
\be
\label{GSCase1Type2}
\bm{\xi}_g=\left( m_1/2,\;  0,\; 0,\; 0,\; m_2/2\right)^T;
\ee
and if $M=0$,   $\bm{\xi}_g$ can be taken in two different types as 
\be\label{GSCase1Type11}
\begin{split}
&\bm{\xi}_g=\left(\gamma_1\cos\theta, \; \gamma_1\sin\theta,\;  \gamma,\;  
-\gamma_1\sin\theta,  \; \gamma_1\cos\theta \right )^T,\quad \hbox{\rm or} \\
&\bm{\xi}_g=\left(\cos\theta/\sqrt{2},\; \sin\theta/\sqrt{2},\; 0, \; \sin\theta/\sqrt{2},\; -\cos\theta/\sqrt{2} \right)^T;
\end{split}
\ee
where $\gamma_1=\sqrt{\fl{1-\gamma}{2}}$ for {\color{black}any $\gamma$ with} $|\gamma|\leq1$ and  $\theta\in[0, 2\pi)$.
\end{lem}

\medskip

\begin{lem}
\label{lem:gs_cyc} 
When $\beta_1>0$ and $\beta_2>0$, 
\textcolor{black}{the ground state $\bm{\xi}_g$ is in cyclic phase.}
Moreover, if $M\in [0,1]$, then $\bm{\xi}_g$ can be taken in three different types as
\be\label{GSCase3Type311}
\begin{split}
&\bm{\xi}_g=\left( m_1^2/4,\; 0, \; \sqrt{2}\,m_1m_2/4,\; 0,\;  m_2^2/4 \right)^T,\quad \hbox{\rm or} \\
&\bm{\xi}_g=\left(\sqrt{3}\,m_3 m_4/4,\; m_3^2/2,\; -\sqrt{2}\,m_3 m_4/4,\; m_4^2/2,\; \sqrt{3}\,m_3 m_4/4\right)^T,
\quad \hbox{\rm or}
\end{split}
\ee
\be
\label{GSCase3Type3}
\bm{\xi}_g=(\xi^g_2,\xi^g_1,\xi^g_0,\xi^g_{-1},\xi^g_{-2})^T,
\ee
where $m_3=\sqrt{1+M}, \  m_4=\sqrt{1-M}$ and
\be\label{GSCase3Type33}
\left\{
\begin{array}{l}
\xi^g_0=-\fl{3\sqrt{6}}{8}M\sin^2\theta\cos\theta
\mp-\fl{\sqrt{2}}{8}\big(2\cot(2\theta)+\cot\theta\big)g(\theta), 	\\[0.5em]
\xi^g_1=\fl{3}{4}M\sin^3\theta+\fl{1}{4}m_2^2\sin\theta
\pm\fl{\sqrt{3}}{4}g(\theta),  	\qquad\quad	\xi^g_{-1}=\sin\theta-\xi^g_1,\\[0.5em]
\xi^g_2=\fl{1}{8}\big(3M\sin^2\theta+2m_1^2\big)\cos\theta
\mp\fl{\sqrt{3}}{8} g(\theta)\tan\theta,  \qquad		\xi^g_{-2}=\xi^g_2-\cos\theta,
\end{array}
\right.
\ee
with $g(\theta)=\sqrt{ \big(m_1m_2-3M^2\sin^2\theta\big) \sin^2\theta\cos^2\theta}$ for 
$\theta\in(0,2\pi)$ satisfying $|\sin\theta|\neq 1$ and
$m_1m_2-3M^2\sin^2\theta\geq 0$.
 If $M\in (1,2)$,  $\bm{\xi}_g$ can only be taken as 
\eqref{GSCase3Type3}.
\end{lem}

\medskip

\begin{remark}
\label{RM2.1} 
For a spin-2 BEC in the spatially uniform system, i.e. under the ansatz 
\eqref{Uniform-GS}, the minimization problem \eqref{groundnorm} collapses to
\be\label{tldbmxi}
\tilde{\bm{\xi}}_g :=\arg \min_{|\bm{\xi}|=1} \mathcal{E}_U(\bm{\xi})=
\arg \min_{M\in[-2,2]}\mathcal{E}_U(\bm{\xi}_g^M), \quad \hbox{\rm with}\quad 
\bm{\xi}_g^M:=\arg \; \min_{\bm{\xi} \in \mathbb{S}_C} \mathcal{E}_U(\bm{\xi}).
\ee
Define 
\be\label{betaM123}
\beta(M) :=\mathcal{E}_U(\bm{\xi}_g^M), \qquad M\in[-2,2].
\ee
By the results in Theorem 2.1, we have for $M\in[-2,2]$ 
\begin{equation}
\label{energy-uniform}
\beta(M):=\mathcal{E}_U(\bm{\xi}_g^M)
 =\fl{1}{2}\left\{\begin{array}{ll} 
\beta_0+4\beta_1, &   \beta_1<0 \;\;  \& \;\; \beta_2>20\beta_1,\\[0.25em]
\beta_0+\fl{\beta_2}{5}+\fl{\left(20\beta_1-\beta_2 \right)M^2}{20}, &   \beta_2<0  \;\;  \& \;\; \beta_2<20\beta_1,\\[0.25em]
\beta_0+\beta_1M^2, &  \beta_1>0  \;\;  \& \;\; \beta_2>0.
\end{array}\right.
\end{equation}
Therefore, the energy of the ground state $\tilde{\bm{\xi}}_g$ defined in
the minimization problem \eqref{tldbmxi} is given as
\begin{equation}
\label{energy-uniform-one-constrain}
\mathcal{E}_U(\tilde{\bm{\xi}}_g) =\min_{M\in[-2,2]}\beta(M)
 =\fl{1}{2}\left\{\begin{array}{ll} 
\beta_0+4\beta_1, &   \beta_1<0 \;\;  \& \;\; \beta_2>20\beta_1,\\
\beta_0+\beta_2/5,&   \beta_2<0  \;\;  \& \;\; \beta_2<20\beta_1,\\
\beta_0, &  \beta_1>0  \;\;  \& \;\; \beta_2>0.
\end{array}\right.
\end{equation}
This immediately suggests that, for the ground state of a spin-2 BEC defined as the minimizer of the energy functional under the total mass conservation, the ground state energy in a spatially uniform system is achieved in nematic and cyclic phases when $M=0$, and respectively, in the ferromagnetic phase for any $M\in[-2,2]$.
The ground state is not unique. In fact, in the ferromagnetic phase, 
the ground state $\tilde{\bm{\xi}}_g$ can be taken as 
\eqref{GSCase2Type1} for any $M\in [0,2]$, and in nematic and cyclic phases, $\tilde{\bm{\xi}}_g$ 
can be taken as \eqref{GSCase1Type11} and 
\eqref{GSCase3Type311}-\eqref{GSCase3Type3} with $M=0$, respectively.
\end{remark}

\subsection{Single mode approximation of ground states}
\label{SMA}
In the literature \cite{BCY,KU2}, the single mode approximation (SMA) is an interesting and 
useful tool for obtaining approximate ground states of spinor BEC. It can reduce 
solving the ground state of a spin-2 BEC to solving the ground state of a single component 
BEC. In fact, in the SMA for a spin-2 BEC, one assumes an ansatz for $\Phi\in {\mathbb S}$ 
in \eqref{ground} as
\be
\label{SMA_GS}
\Phi(\bx)=\phi(\bx)(\xi_2,\;\xi_1,\;\xi_0,\;\xi_{-1},\;\xi_{-2})^T =:\phi(\bx)\,\bm{\xi}=:
\Phi_{\rm sma}(\bx),
\ee
where $\bm{\xi}=(\xi_2,\;\xi_1,\;\xi_0,\;\xi_{-1},\;\xi_{-2})^T\in\mathbb{S}_C$   
and $\phi:=\phi(\bx)\in  \widetilde{ \mathbb{S}}_1 :=\{\varphi | \int_{\mathbb{R}^d}|\varphi(\bx)|^2 d\bx=1\}.$  
 Plugging \eqref{SMA_GS} into \eqref{energy}, noticing \eqref{energy_uniform},
we obtain 
\be
\label{energy_sma}
\mathcal{E}(\Phi)=\mathcal{E}(\Phi_{\rm sma})
 =\int_{\mathbb{R}^d} \left[\frac{1}{2}|\nabla\phi|^2+V|\phi|^2 
+\mathcal{E}_U(\bm{\xi})|\phi|^4 \right]d\bx =:\mathcal{E}_{\rm sma}(\phi,\bm{\xi}).
\ee
Then the SMA of  ground state is to find $\Phi^g_{\rm sma}=\phi_g\,\bm{\xi}_g$
with $\bm{\xi}_g\in\mathbb{S}_C$ and $\phi_g\in \widetilde{ \mathbb{S}}_1$ s.t. 
\be\label{Phismag}
\Phi^g_{\rm sma}:=\arg\;\min_{\Phi_{\rm sma}\in\mathbb{S}}\mathcal{E}(\Phi_{\rm sma}).
\ee
Combining \eqref{Phismag}, \eqref{energy_sma}, \eqref{tldbmxi} and \eqref{betaM123}, we get
\bes
\begin{split}
\mathcal{E}(\Phi^g_{\rm sma})=&\min_{\Phi_{\rm sma}\in\mathbb{S}}\mathcal{E}(\Phi_{\rm sma})=
\min_{\phi\in\widetilde{ \mathbb{S}}_1}\min_{\bm{\xi}\in\mathbb{S}_C} \mathcal{E}_{\rm sma}(\phi,\bm{\xi})\\
=&\min_{\phi\in \widetilde{ \mathbb{S}}_1}\left\{\int_{\mathbb{R}^d} \left[\frac{1}{2}|\nabla\phi|^2
+V|\phi|^2+\left(\min_{\bm{\xi}\in\mathbb{S}_C}\mathcal{E}_U(\bm{\xi})\right)\, |\phi|^4 \right]d\bx
\right\}
=\min_{\phi\in\widetilde{ \mathbb{S}}_1}\;E_{\rm sma}(\phi),
\end{split}
\ees
where
\be
E_{\rm sma}(\phi)=
\int_{\mathbb{R}^d} \left[\frac{1}{2}|\nabla\phi|^2
+V|\phi|^2+\beta(M)\, |\phi|^4 \right]d\bx.
\ee
Thus $\bm{\xi}_g$ is the ground state of a 
spin-2 BEC in a spatially uniform system, i.e. \eqref{ground-uniform}, and 
$\phi_g$ is the ground state of a singe-component BEC, i.e.
\be\label{phig1c}
\phi_g:=\arg\; \min_{\phi\in\widetilde{ \mathbb{S}}_1}\;E_{\rm sma}(\phi).
\ee

As it has been observed numerically and proved mathematically 
in the literature \cite{BCY,KU2},  the above single mode approximation of  ground state  indeed  gives the ground state of the spin-2 BEC in the following cases: (i) when $M=\pm2$, (ii) when $M=0$, and (iii) in the ferromagnetic phase when $\beta_1<0$ and $\beta_2>20\beta_1$ for $M\in[-2,2]$. Thus, in these cases, the computation of ground state of a spin-2 BEC can be reduced to the computation of the ground state of a single component BEC, i.e. \eqref{phig1c}.  Certainly, for all the other cases, one has to solve the original minimization problem  \eqref{ground}.

\section{An efficient and accurate numerical method}
In this section, we first present a normalized gradient flow (NGF) to compute the 
ground state of the spin-2 BEC, then introduce three additional 
equations for determining the five projection constants in the projection step
for semi-discretization of the NGF in time, and finally a 
full discretization of the NGF is proposed.

\subsection{A continuous normalized gradient flow (CNGF)}
In order to compute the ground state of spin-2 BEC \eqref{ground}, similar as 
for the single-component BEC \cite{BCY0,BD} and spin-1 BEC \cite{BCY0,BL}, here we first 
present a continuous normalized gradient flow (CNGF) for $\Phi:=\Phi(\bx,t)=(\phi_2,\phi_1,\phi_0,\phi_{-1},\phi_{-2})^T:=
(\phi_2(\bx,t),\phi_1(\bx,t),\phi_0(\bx,t)$, $\phi_{-1}(\bx,t),\phi_{-2}(\bx,t))^T $ as
\cite{W2}:   
\be 
\label{eq: CNGF1}
\qquad\p_t\phi_\ell=-[\textcolor{black}{H_\rho}+a_\ell(\Phi)]\phi_\ell-f_\ell(\Phi)
+[\mu_\Phi(t)+\ell\lambda_\Phi(t)]\phi_\ell =: \big(\BH\Phi\big)_{\ell}, \ -2\le \ell\le 2,
\ee
where $a_\ell:=a_\ell(\Phi)$ and $f_\ell:=f_\ell(\Phi)$ ($\ell=2,\cdots,-2$) are given as
\[
\begin{split}
a_0=&3\beta_1(|\phi_1|^2+|\phi_{-1}|^2)+0.2\beta_2|\phi_{0}|^2,\\[0.5em]
 f_0=&\beta_1\bigg[\bigg(\fl{\sqrt{6}}{2}F_{+}-3\bar{\phi}_1\phi_0\bigg)\phi_1+
\bigg(\fl{\sqrt{6}}{2}F_{-}- 3\bar{\phi}_{-1}\phi_0\bigg)\phi_{-1}\bigg]
+\beta_2\bigg(\fl{A_{00}}{\sqrt{5}}+0.2\beta_2\phi_0^2\bigg)\bar{\phi}_0,\\[0.5em]
 a_{\pm\ell} =&\beta_1\big[2\,|\phi_{\pm(3-\ell)}|^2\pm\ell\,F_z(\Phi)+(6-3\ell)\,|\phi_0|^2\big]+0.4\beta_2|\phi_{\mp\ell}|^2,\qquad \ell=2,1\\[0.5em]
 f_{\pm\ell}=&\beta_1\bigg\{ \Big[(\ell-1)F_\mp+(2-\ell)F_\pm -2\,\bar{\phi}_{\pm(3-\ell)}\phi_{\pm\ell}\Big]\phi_{\pm(3-\ell)}\\[0.05em]
 &+(2-\ell)\,\bigg[\fl{\sqrt{6}}{2}F_\mp-3\bar{\phi}_0\phi_{\pm\ell}\Big]\phi_0\bigg\}
 +\beta_2\bigg[\fl{(-1)^\ell \,A_{00}}{\sqrt{5}}-0.4\phi_{\pm\ell}\phi_{\mp\ell}\bigg]\bar{\phi}_{\mp\ell},\quad \ell=2,1.
 \end{split}
\]
Here $\mu_\Phi(t)$ and $\lambda_\Phi(t)$  are the Lagrangian multipliers such that both the mass 
\eqref{mass} and magnetization \eqref{magnet} are conserved during  dynamics and they can be taken as \cite{W2}: 
\begin{equation}\label{h213}
\left\{
\begin{split}
&\mu_\Phi(t)=\frac{\mathcal{R}(\Phi(\bx,t)) \,\mathcal{K}(\Phi(\bx,t))-\mathcal{M}(\Phi(\bx,t))\, 
\mathcal{P}(\Phi(\bx,t))}
{\mathcal{R}(\Phi(\bx,t))\, \mathcal{N}(\Phi(\bx,t))- \mathcal{M}^2(\Phi(\bx,t))},\\[0.5em]
&\lambda_\Phi(t)=\frac{\mathcal{N}(\Phi(\bx,t))\, \mathcal{P}(\Phi(\bx,t))-\mathcal{M}(\Phi(\bx,t))\, \mathcal{K}(\Phi(\bx,t))}
{\mathcal{R}(\Phi(\bx,t))\, \mathcal{N}(\Phi(\bx,t))-\mathcal{M}^2(\Phi(\bx,t))},
\end{split}
\right.
\end{equation}
with $\mathcal{N}(\Phi(\bx,t))$ and $\mathcal{M}(\Phi(\bx,t))$ given in  \eqref{mass}  and \eqref{magnet}, respectively, and 
\[
\begin{split}
&\mathcal{R}(\Phi(\bx,t))=\sum_{\ell=-2}^2 \ell^2\|\phi_{\ell}(\bx,t)\|^2,\quad 
\mathcal{K}(\Phi(\bx,t))=\sum_{\ell=-2}^2\int_{\mathbb{R}^d}\bar{\phi}_\ell(\bx,t)\,
\big(\BH\Phi\big)_{-\ell}(\bx,t)\,d\bx,\\
&\mathcal{P}(\Phi(\bx,t))=\sum_{\ell=-2}^2\int_{\mathbb{R}^d}\,\ell\,\bar{\phi}_\ell(\bx,t)
\,\big(\BH\Phi\big)_{-\ell}(\bx,t)\,d\bx, \qquad t\ge0.
\end{split}
\]
For any given initial data $\Phi(\bx,0):=\Phi_0(\bx)$
satisfying 
\be \label{conserve-initial}
\mathcal{N}(\Phi_0(\bx))=1, \qquad\qquad
\mathcal{M}(\Phi_0(\bx))=M,
\ee
it is easy to show that the CNGF \eqref{eq: CNGF1} conserves the total 
mass and magnetization meanwhile  diminishes the total energy \cite{W2}, i.e., 
\be
\begin{split}
&\mathcal{N}(\Phi(\cdot,t)) \equiv \mathcal{N}(\Phi_0)=1, 
\qquad \mathcal{M}(\Phi(\cdot,t)) \equiv \mathcal{M}(\Phi_0)=M, \qquad t\ge0, \\
&\mathcal{E}(\Phi(\cdot,t))\leq \mathcal{E}(\Phi(\cdot,s))\leq \mathcal{E}(\Phi_0), \qquad {\rm for}\quad \forall\; t\ge s\ge 0.
\end{split}
\ee
Thus the ground state of spin-2 BEC \eqref{ground} can be obtained 
as the steady state of the CNGF \eqref{eq: CNGF1} with proper choice of the initial data $\Phi_0$ satisfying  
\eqref{conserve-initial}.

\subsection{A gradient flow with discrete normalization (GFDN)}
\label{sec:algorithm}
Choose a time step size {\color{black}$\Delta t>0$} and denote time steps as  {\color{black}$t_n=n\Delta t$} for $n\ge0$.  
Then a gradient flow with discrete normalization (GFDN) for 
computing the ground state of the spin-2 BEC \eqref{ground} 
can be constructed by first  applying the first-order time-splitting 
semi-discretization of the CNGF \eqref{eq: CNGF1} as
\be
\label{eq:GF}
 \partial_t\phi_\ell
=-\left(\textcolor{black}{H_\rho}+a_\ell(\Phi)\right)\phi_{\ell}-f_\ell(\Phi),\quad  t\in[t_{n-1},t_n), \quad
\ell=2,1,\cdots,-2,
\ee
followed by a projection step as
\be
\label{proj}
\phi_\ell(\bx,t_n):=\phi_\ell(\bx,t_n^+)=\sigma_\ell^n\phi_\ell(\bx,t_n^-),\qquad \ell=2,1,\cdots,-2.
\ee
Here  $\phi_\ell(\bx,t_n^{\pm})=\lim_{t\rightarrow t_n^{\pm}}\phi_\ell(\bx,t)$. Moreover, the projection constants $\sigma_\ell^n\ge0$ for $\ell=2,1,\cdots,-2$ are  to be
 chosen such that 
\begin{equation}
\label{h37}
 \|\Phi(\cdot,t_n)\|^2=\sum_{\ell=-2}^2\|\phi_\ell(\cdot,t_n)\|^2=1,\quad\quad
 \sum_{\ell=-2}^2 \ell \|\phi_\ell (\cdot,t_n)\|^2=M.
 \end{equation}
Plugging \eqref{proj} into \eqref{h37}, we have
\begin{equation}
\label{h310}
 \sum_{\ell=-2}^2(\sigma_\ell^n)^2\|\phi_\ell(\cdot,t_n^-)\|^2=1,
 \qquad\qquad \sum_{\ell=-2}^2 \ell (\sigma_\ell^n)^2\|\phi_\ell(\cdot,t_n^-)\|^2=M.
\end{equation}

In fact, in the projection step, we have to determine the five projection constants $\sigma_\ell^n$
for $\ell=2,\cdots, -2$ in \eqref{proj}.  However, we only have two equations in \eqref{h310}.
In order to find additional proper constraints for determining the five projection  constants
in  projection step \eqref{proj}, we can view the GFDN \eqref{eq:GF}-\eqref{proj} as 
a first-order time-splitting semi-discretization of the CNGF \eqref{eq: CNGF1}. 
In this regard, the projection step \eqref{proj} is  {\color{black} similar} to
solving the following nonlinear ordinary differential equations (ODEs):
\be
 \p_t\phi_\ell(\bx,t)=[\mu_\Phi(t)+\ell\lambda_\Phi(t)]\phi_\ell,\quad 
 t_{n-1}\leq t \leq t_{n}, \quad \ell=2,1,\cdots,-2.
\ee
Solving the above ODEs,  one obtains for $\ell=2,1,\cdots,-2$ as
\be
\phi_\ell(\bx,t_n)= \phi_\ell(\bx,t_{n-1})\exp\left( \int_{t_{n-1}}^{t_n} 
[\mu_\Phi(s)+\ell\lambda_\Phi(s)]\ ds \right ):=\tilde\sigma_\ell^n\phi_\ell(\bx,t_{n-1}), 
\ee
which  suggests the following three relationships for the constants $\tilde\sigma_\ell^n\;
 (\ell=2,1,\cdots,-2$): 
 \be
\tilde\sigma_2^n\tilde\sigma_{-2}^n=(\tilde \sigma_0^n)^2,\qquad \tilde\sigma_1^n\tilde\sigma_{-1}^n=(\tilde\sigma_0^n)^2,\qquad
\tilde\sigma_2^n\tilde\sigma_{0}^n=(\tilde\sigma_1^n)^2.
\ee
Based on the above observation, we propose and adapt
the following three additional constrains for determining the five 
projection constants in the project step \eqref{proj} as
\begin{equation}
\label{eq: Additional_Cond}
 \sigma_2^n\sigma_{-2}^n=(\sigma_0^n)^2,\qquad\quad  \sigma_1^n\sigma_{-1}^n=(\sigma_0^n)^2,\qquad\quad  \sigma_2^n\sigma_{0}^n=(\sigma_1^n)^2.
 \end{equation}
 
For the existence and uniqueness (in most cases) of the five project constants $\sigma_\ell^n\ge0$ for $\ell=2,1,\cdots,-2$
governed by \eqref{h310} and \eqref{eq: Additional_Cond}, we have the following result.

 \medskip
 
\begin{theor}
\label{TH3.1}
For sufficiently small time step size {\color{black}$\Delta t>0$}
and for $\forall$  $M\in [0,2)$, the solution $\sigma_\ell^n$ ($\ell=2,\cdots,-2$)
satisfying \eqref{h310} and \eqref{eq: Additional_Cond} can be solved out as:

\smallskip

\begin{itemize}
\item[(i).] when $M=1$, $\|\phi_{1}(\bx,t_n^{-})\|>0$ and $\|\phi_{\ell}(\bx,t_n^{-})\|=0 \ (\ell=2,0,-1,-2)$, then
\be 
\label{constant_P_1}
\sigma_{1}^n=\frac{1}{\|\phi_{1}(\bx,t_n^{-})\|}, 
\qquad  \sigma_0^n=1, \qquad  \sigma_\ell^n=(\sigma_{1}^n)^\ell, \quad \ell=2,-1,-2;
\ee

\item[(ii).] when $M=0$, $\|\phi_{0}(\bx,t_n^{-})\|>0$ and $\|\phi_{\ell}(\bx,t_n^{-})\|=0 \  (\ell=2,1,-1,-2)$, then
\be 
\label{constant_P_2}
\sigma_{0}^n=\frac{1}{\|\phi_{0}(\bx,t_n^{-})\|},  \qquad \sigma_1^n=1, \qquad 
\sigma_\ell^n=(\sigma_0^n)^{1-\ell}, \quad \ell=2,-1,-2; 
\ee

\item[(iii).] for all other cases, then 
\be 
\label{constant_main}
\sigma_0^n=\fl{1}{\sqrt{\sum_{\ell=-2}^2\lambda_*^{\ell}\|\phi_{\ell}(\bx,t_n^{-})\|^2}},\qquad
\sigma_\ell^n=\sigma_0^n (\lambda_*)^{\ell/2}, 
 \ee
 for $\ell=-2, -1, 1, 2,$. Here, $\lambda_*$ is the unique positive solution of the following fourth-order
algebraic equation with respect to the unknown $\lambda$ as 
\be
\sum_{\ell=-2}^2(\ell-M)\|\phi_{\ell}(\bx,t_n^{-})\|^2 \lambda^{\ell+2}=0.
\ee
\end{itemize}
\end{theor}

\smallskip

\begin{proof}
Combining \eqref{h310} and \eqref{eq: Additional_Cond}, it is straightforward to  
check that \eqref{constant_P_1} is a solution (not unique) of \eqref{h310} and \eqref{eq: Additional_Cond}
in case (i), and \eqref{constant_P_2} is a solution (not unique) of \eqref{h310} 
and \eqref{eq: Additional_Cond} in case (ii).

Now we prove \eqref{constant_main} in case (iii). Noticing \eqref{h37} is also valid when $t_n$ is replaced by $t_{n-1}$ 
for $n\ge1$, i.e.
\be
\label{eq:appd_equaL}
\sum_{\ell=-2}^2 \|\phi_\ell(\bx,t_{n-1})\|^2=1, \qquad  
 \qquad \sum_{\ell=-2}^2 \ell \|\phi_\ell(\bx,t_{n-1})\|^2=M,
\ee
we have
\be
\label{Q_eq0}
\sum_{\ell=-2}^2(\ell-M) \|\phi_\ell(\bx,t_{n-1})\|^2=0.
\ee
In addition, for sufficiently small time step size {$\color{black}\Delta t>0$}, $\|\phi_\ell(\cdot,t)\|
\in\mathcal{C}([t_{n-1},t_n))$ $ (\ell=2\cdots,-2)$ implies   
\be
\label{eq: Apend_E}
\|\phi_\ell(\bx,t_{n-1})\|>0  \qquad \Longleftrightarrow \qquad \|\phi_\ell(\bx,t_n^{-})\|>0.
\ee
Combining \eqref{eq: Apend_E} and \eqref{eq:appd_equaL}, we get $\sum_{\ell=-2}^2 \|\phi_\ell(\bx,t_n^{-})\|^2 >0$. Therefore, case $(iii)$ can be divided into three sub-cases: 

\begin{itemize}
\item[(a)] $M=1$ and $\|\phi_2(\bx,t_n^{-})\|+\|\phi_0(\bx,t_n^{-})\|+\|\phi_{-1}(\bx,t_n^{-})\|+\|\phi_{-2}(\bx,t_n^{-})\|>0$.

\item[(b)] $M=0$ and $\|\phi_2(\bx,t_n^{-})\|+\|\phi_1(\bx,t_n^{-})\|+\|\phi_{-1}(\bx,t_n^{-})\|+\|\phi_{-2}(\bx,t_n^{-})\|>0$. 

\item[(c)] $M\neq1$ and $M\neq 0$, thus $\ell-M\neq 0 \ (\ell=2,\cdots,-2)$  by noticing $M\in[0, 2)$.
\end{itemize}
For all the above three sub-cases (a)-(c), noting \eqref{Q_eq0}-\eqref{eq: Apend_E}, 
 it holds  
 \be
\label{Q_eq2}
  \sum_{\ell<M} \|\phi_\ell(\bx,t_n^{-})\|^2 >0, \qquad  \qquad 
  \sum_{\ell>M} \|\phi_\ell(\bx,t_n^{-}) \|^2>0.
\ee
Denote  $\lambda=\sigma_2^n/\sigma_0^n$,  by \eqref{h310} and \eqref{eq: Additional_Cond},
for any $M\in[0, 2)$,  we have
\be\label{Q_cofficient}
\left\{
\begin{split}
&\sigma_0^n=\frac{1}{\sqrt{\sum_{\ell=-2}^2 \|\phi_\ell(\bx,t_n^-)\|^2 \lambda^\ell}},
\qquad \sigma_j^n=\sigma_0^n \lambda^{j/2},\quad j=2,1,-1,-2, \\
&g(\lambda):=\sum_{\ell=-2}^2(\ell-M)\|\phi_\ell(\bx,t_n^-)\|^2 \lambda^{\ell+2} =0.
\end{split}
\right.
\ee
Then we need only to show that $g(\lambda)$ has a unique positive root to finish the proof.

Define
$\ell_1^n$, $\ell_2^n\in \{2,1,0,-1,-2\}$ as
\[
\ell_1^n:=\min \{ \ell \;  |   \;   \|\phi_\ell(\bx,t_n^{-})\| \neq 0 \}<
\ell_2^n:=\max \{ \ell \; |  \; \|\phi_\ell(\bx,t_n^{-})\| \neq 0 \}.
\]
Then  \eqref{Q_eq2} indicates  that   $\ell_1^n<M$ and $\ell_2^n>M$. 
Hence $g(\lambda)$  can be reformulated  as 
 \[\begin{split}
g(\lambda)&=\sum_{\ell_1^n\leq \ell<M}(\ell-M)\|\phi_\ell(\bx,t_n^-)\|^2 \lambda^{\ell+2}
+\sum_{M<\ell \leq \ell_2^n}(\ell-M)\|\phi_\ell(\bx,t_n^-)\|^2  \lambda^{\ell+2}\\
&=:h(\lambda) \lambda^{\ell_1^n+2},
\end{split}
\]
where 
\[\begin{split}
h(\lambda)&=\sum_{\ell_1^n\leq \ell<M}(\ell-M)\|\phi_\ell(\bx,t_n^-)\|^2 \lambda^{\ell-\ell_1^n} 
 +\sum_{M<\ell \leq \ell_2^n}(\ell-M)\|\phi_\ell(\bx,t_n^-)\|^2  \lambda^{\ell-\ell_1^n}\\
 &=:h_1(\lambda)+h_2(\lambda).
\end{split}
\]
A simple calculation shows 
\bes
\lim_{\lambda \rightarrow 0^+} h(\lambda)=(\ell_1^n-M)\|\phi_{\ell_1^n}(\bx,t_n^-)\|^2<0, \qquad  \qquad\lim_{\lambda \rightarrow +\infty} h(\lambda)=+\infty, 
\ees
which immediately  implies that $h(\lambda)$  has  at least one positive root $\lambda_*>0$. 
In addition, at any positive root $\lambda=\lambda_*>0$ of $h(\lambda)$, noticing \eqref{Q_eq2}, we have 
\[
\begin{split}
h'(\lambda_*)=&h_1'(\lambda_*)+h_2'(\lambda_*)\\
=&\sum_{\ell_1^n\leq \ell<M} \left(\ell-\ell_1^n \right)(\ell-M)\|\phi_\ell(\bx,t_n^-)\|^2 (\lambda_*)^{\ell-\ell_1^n-1}+h_2'(\lambda_*)\\
=&\lambda_*^{-1} \;(M-\ell_1^n) \sum_{\ell_1^n\leq \ell<M}(\ell-M)\|\phi_\ell(\bx,t_n^-)\|^2 
(\lambda_*)^{\ell-\ell_1^n}  +h_2'(\lambda_*) \\
=&\lambda_*^{-1} \;(M-\ell_1^n) \; h_1(\lambda_*)+h_2'(\lambda_*)   \\
=&\lambda_*^{-1} \;(M-\ell_1^n) \; \big[h(\lambda_n)-h_2(\lambda_*)\big]+h_2'(\lambda_*)  
=\lambda_*^{-1} \;(\ell_1^n-M)\;h_2(\lambda_*) +h_2'(\lambda_*) \\
=&\sum_{M<\ell \leq \ell_2^n} \left(\ell_1^n-M\right)  (\ell-M)\|\phi_\ell(\bx,t_n^-)\|^2 
(\lambda_*)^{\ell-\ell_1^n-1}  +h_2'(\lambda_*) \\
=&\sum_{M<\ell \leq \ell_2^n} \big[ (\ell-M)^2 -\left(\ell-\ell_1^n \right)(\ell-M) \big]\;
\|\phi_\ell(\bx,t_n^-)\|^2 (\lambda_*)^{\ell-\ell_1^n-1} +h_2'(\lambda_*)\\
=&\sum_{M<\ell \leq \ell_2^n} (\ell-M)^2\|\phi_\ell(\bx,t_n^-)\|^2 (\lambda_*)^{\ell-\ell_1^n-1}>0.
\end{split}
\]
Therefore $h(\lambda)$  (and thus $g(\lambda)$)   has exactly one positive root $\lambda_*$. 
Substituting  $\lambda=\lambda_*$ into \eqref{Q_cofficient} leads to the formulas for 
the projection constants in \eqref{constant_main}.    \hfill 
\end{proof}

\subsection{A backward-forward Euler finite difference discretization}

Due to the trapping potential $V(\bx)$,  the solution $\Phi(\bx,t)$ of 
the CNGF \eqref{eq: CNGF1} (or the GFDN \eqref{eq:GF}-\eqref{proj})
decays exponentially fast to zero as $| \bx | \rightarrow \infty$. 
Hence, one can  truncate the problem into a bounded domain $\mathcal{D}$ with the homogeneous Dirichelet   boundary condition in practical computation. 
Various methods such as the backward (-forward) Euler finite difference/sine-spectral  method \cite{BCL,BD,BL}    can be applied to discretize the GFDN \eqref{eq:GF}-\eqref{proj}.  In this section, we adapt a 
backward-forward Euler finite difference method (BEFD) to discretize the GFDN \eqref{eq:GF}-\eqref{proj}.

To simplify the presentation, we introduce the scheme for the case of one
spatial dimension, i.e. $d=1$, defined on an interval $\mathcal{D}=\left(a,b \right)$ with
the homogeneous Dirichelet  boundary condition. Generalization to higher dimension is straightforward by tensor product. 
For $d=1$,  the spatial mesh size is chosen as  $h=\left(b-a\right)/N$  with $N$ an even positive integer.  Let 
$
x_j:=a+j h, \quad j=0,\cdots,N
$ 
be the grid points, denote respectively $\Phi_j^n$ and \textcolor{black}{$\rho_j^n$} 
as the approximation of $\Phi(x_j,t_n)$ and \textcolor{black}{$\rho(x_j,t_n)$}.  Moreover, we denote $\Phi^n$ as the solution vector with component $\Phi_j^n$.  Then the GFDN \eqref{eq:GF}-\eqref{proj} is 
discretized  as 
\be\label{eq:scheme-GF}
\left\{
\begin{split}
&\frac{\phi^\ast_{\ell,j}-\phi_{\ell,j}^{n}}{{\color{black}\Delta t}}=\left(\fl{1}{2}\delta_h^2-V_j-\textcolor{black}{\rho_j^n}-a_\ell(\Phi_j^{n}) \right)\phi_{\ell,j}^\ast-f_\ell(\Phi_j^{n}), \quad  1\le j\le N-1,\\
&\phi_{\ell,j}^{n+1}=\sigma_\ell^n\, \phi_{\ell,j}^\ast, \qquad \qquad \ell=2,\cdots,-2, \qquad  j=0,1,\cdots, N.
\end{split}
\right.
\ee
Here, $V_j=V(x_j)$,  $\delta_h^2$ is the second-order central finite difference operator and  $\sigma_\ell^n (\ell=2,\cdots,-2)$ are the projection constants chosen as   \eqref{constant_P_1}-\eqref{constant_main}. In addition, the homogeneous Dirichelet  boundary condition and initial data  are  discretized as  
\be
\phi_{\ell,0}^\ast=\phi_{\ell,N}^\ast=0,\qquad \phi_{\ell,j}^0=\phi_{\ell}(x_j, 0),
\qquad \ell=2,\cdots,-2,		\quad j=0,1,\cdots,N.
\ee


\section{Numerical results}
In this section, we first study how to choose proper initial data for computing numerically the ground states
of spin-2 BECs, then apply the numerical method to compute the ground states
under different interaction parameters $\beta_1$ and $\beta_2$ as well as 
the magnetization $M$ in one and two dimensions. Uniqueness and non-uniqueness of  the ground state 
are tested and discussed based on our extensive numerical results. 
In our numerical computations,  the ground state $\Phi_g:=\lim_{n\to\infty}\Phi^n$ 
is reached numerically when $\frac{\|\Phi^{n+1}-\Phi^n\|_{\infty}}{\textcolor{black}{\Delta t}} \leq \varepsilon :=10^{-7}$.

In practice, unless stated, we fix $\beta_0=100,$  ${\color{black}\Delta t}=0.005$ and  
$\mathcal{D}=[-10,10]^d$ for $d=1,2$. The mesh size is taken  as 
$h_x=1/64$ when $d=1$, and respectively, $h_x=h_y=1/16$ when $d=2$.
Moreover, $V(\bx)$ is chosen either as  the
harmonic plus optical lattice potential 
\be
\label{num_choice}
V(\bx)=\sum_{j=1}^d\left[\fl{1}{2}\nu^2_j+\eta\,(d-1)\sin^2\left(q_j\nu_j\right)\right], 
\quad  \bx\in\mathcal{D}, \quad d=1,2,
\ee
 with $\nu_1=x$, $\nu_2=y$, $\eta$ and $q_j\; (j=1,2)$ given constants, or the  box potential 
\begin{equation}
\label{box_potential}
V_{\rm box}(\bx)=\left\{\begin{array}{ll} 
0, &  \bx\in\mathcal{D},\\[0.25em]
+\infty, &  \hbox{\rm otherwise}.
\end{array}\right.
\end{equation}

\subsection{Choice of initial data and uniqueness of the ground state}
\label{sec:inidata_choice}
A  proper choice of   initial data $\Phi_0(\bx)$  usually improves significantly 
the efficiency and accuracy of the GFDN \eqref{eq:GF}-\eqref{proj}. For cases where SMA
is valid, e.g. the nematic  phase with $M=0$  and the ferromagnetic phase \cite{BCY},
see also Fig.~\ref{fig_sma}), one can either simply construct  ground state via \eqref{phig1c} 
 by solving the  ground state of the  single component BEC \eqref{phig1c},
or directly solve the ground state via the GFDN 
\eqref{eq:scheme-GF}-\eqref{proj}  with  initial data (it is a reasonable
and natural choice by noticing  section \ref{SMA}) taken as
\be 
\label{ini_choice}
\bm{\Phi}_0(\bx)=\bm{\xi}_g\,\phi(\bx),
\ee
where  $\bm{\xi}_g$  is given in \textcolor{black}{lemmas \ref{lem:gs_ferr}-\ref{lem:gs_cyc}}  and 
$\phi(\bx)$ is an approximation of the ground state of 
the single component BEC \eqref{phig1c}, e.g., the harmonic oscillator approximation $\phi_g^{\rm hos}(\bx)$ for 
small $\beta(M)$ and/or the  Thomas-Fermi approximation $\phi_g^{\rm TF}(\bx)$ for large $\beta(M)$ \cite{BD, BL}.  
For other cases where the SMA is invalid, the initial data $\Phi_0$ can be chosen either as 
\eqref{ini_choice} or as a more general initial set-up 
\be
\label{xi_choice}
\bm{\Phi}^0(\bx)= \phi(\bx)\bm{\xi}=:\fl{\phi(\bx)}{2}\Big( \sqrt{2+M-2\sigma},\,\sqrt{\sigma},\, \sqrt{2\sigma},\,
 \sqrt{\sigma},\, \sqrt{2-M-2\sigma} \Big)^T,
\ee
with $\sigma\in[0, 1-M/2]$. 
Extensive numerical comparison (not shown here for brevity) for different  initial data
show that the  GFDN would usually converge faster with initial data \eqref{ini_choice} than with
other types of initial data. Based on those comparison, we would conclude and suggest 
the choice of initial data as follows:

a). For ferromagnetic phase,  $\bm{\xi}_g$ is suggested to be chosen as \eqref{GSCase2Type1}  for  
$\forall M\in[0, 2)$.  Meanwhile, the ground state is  found to be  unique. 

b). For nematic phase, if $M\in(0, 2)$, $\bm{\xi}_g$ is suggested to be chosen as \eqref{GSCase1Type2},
 the ground state is  found to be unique. 
  However, if $M=0$,  the ground state is not unique. Hence, 
  any initial data chosen as \eqref{GSCase1Type2} works and
 probably converge to different ground states.

c). For cyclic phase, if $M\in(0, 2)$, then  $\bm{\xi}_g$ is suggested to be chosen as  
\eqref{GSCase3Type3} with 
$\theta=\arctan\sqrt{(2-M)/(1+M)}$, i.e.,
\begin{equation}
\label{h414}
\bm{\xi}_g=\left(\sqrt{(M+1)/3}, \ 0,\ 0,\ \sqrt{(2-M)/3}, \ 0 \right)^T.
\end{equation}
The ground state is  found to be  unique, which is essentially  different from the 
spatially uniform system where the ground state is not unique. While if $M=0$, similar as the nematic phase,  
the ground state is not unique, thus any initial data works and probably converge to different ground states.


\begin{figure}[htbp!]
\centerline{
\psfig{figure=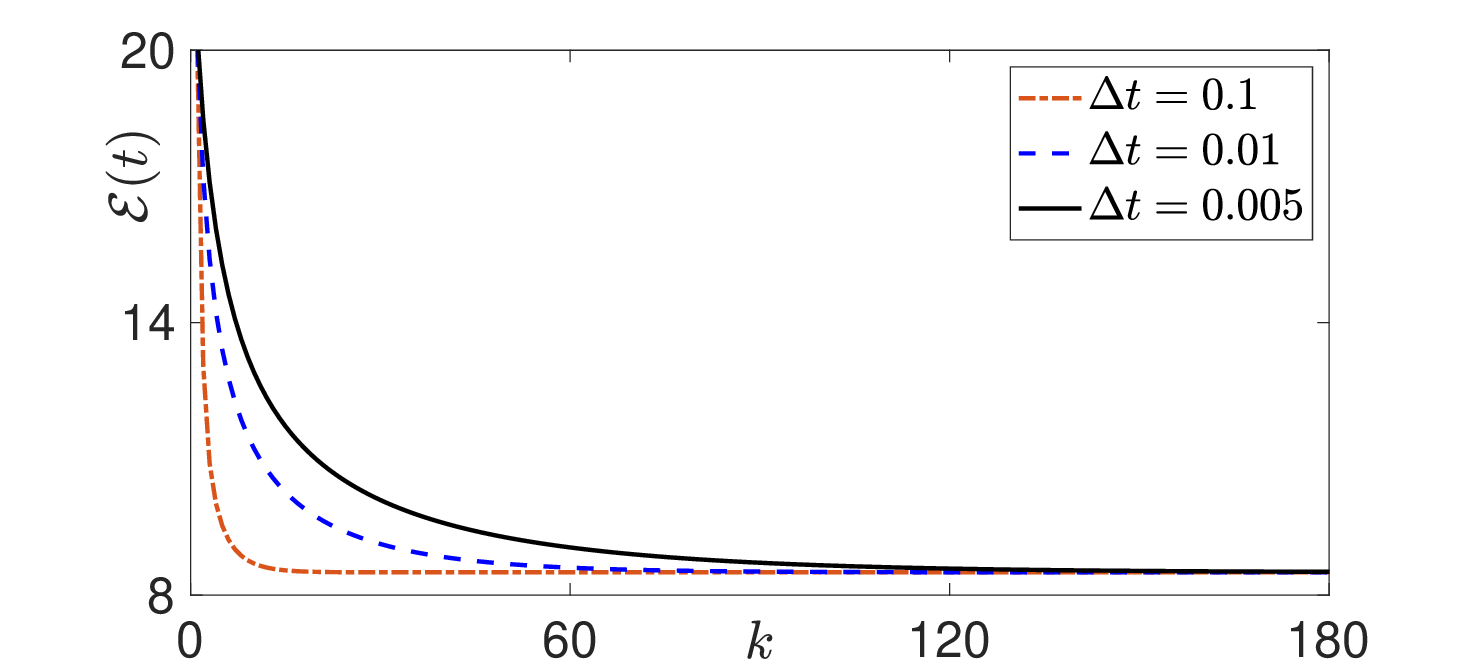,width=11.5cm,height=4cm}
}
{
\caption{Time evolution of the total energy $\mathcal{E}(t)$  in 
Example \ref{eg:choice_ini_data} with different time step size ${\color{black}\Delta t}$.}
}
\label{fig1} 
\end{figure}

 \medskip
 
 \begin{exmp}
 \label{eg:choice_ini_data}
 Here we show the energy-diminishing property of our numerical method.
 To this end, we let $M=0.5$,  $\beta_1=1$ and $\beta_2=-2$ (i.e. nematic phase).
 The ground state of a spin-2 BEC  is computed  by the BEFD
 \eqref{eq:scheme-GF}  with time step ${\color{black}\Delta t}=0.1/0.01/0.005$ and the initial data 
 \eqref{xi_choice} with $\sigma=0$, i.e., $\bm{\xi}$ is chosen as the ground state
  in a spatially uniform system \eqref{GSCase1Type2}.
  Fig. \ref{fig1} shows the evolution of the energy $\mathcal{E}(t):=\mathcal{E}(\Phi(\cdot,t))$ 
  with different time step ${\color{black}\Delta t}$.
 \end{exmp}
 
 From Fig.  \ref{fig1}  and additional  experiments not shown here for brevity, we can see that: 
(i). the energy is diminishing for different time step size ${\color{black}\Delta t}$, 
even for the relatively large step size ${\color{black}\Delta t}=0.1$ (cf. Fig. \ref{fig1}), and  (ii). the GFDN with different 
initial data converge to the same ground state. In addition, when $\sigma=0$, i.e.,  
$\bm{\xi}$ is chosen as the ground state in a spatially uniform system \eqref{GSCase1Type2}, the GFDN
usually converges in the fastest way.

\begin{figure}[htbp!]
\centerline{
\psfig{figure=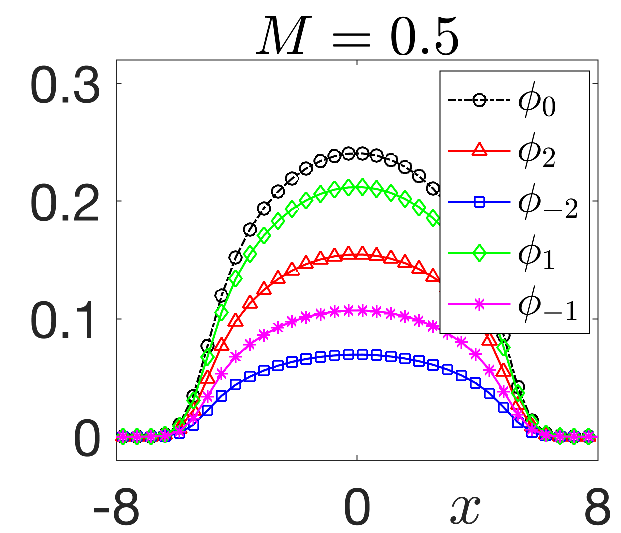,width=4.5cm,height=3.2cm}\qquad\qquad
\psfig{figure=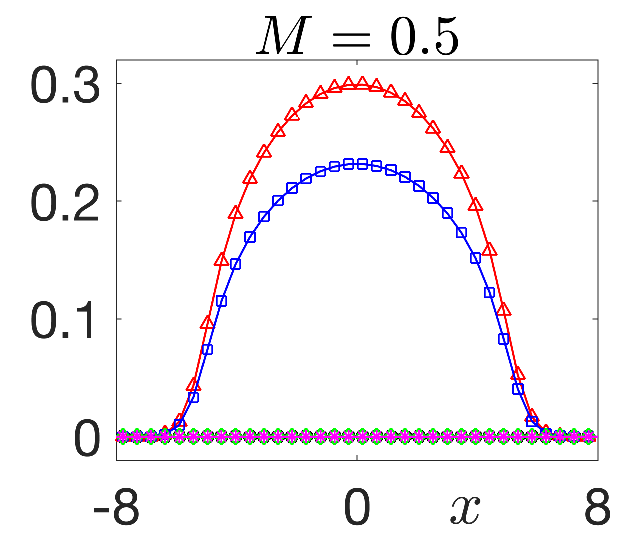,width=4.5cm,height=3.2cm}
}
\vspace{0.25cm}
\centerline{
\psfig{figure=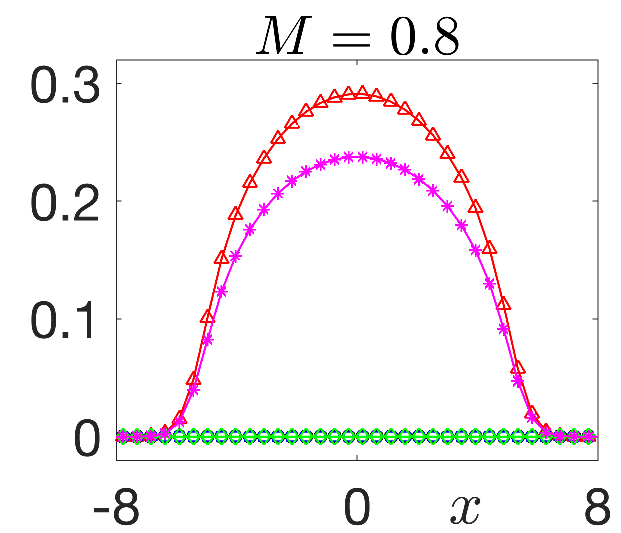,width=4.5cm,height=3.2cm}\qquad\qquad
\psfig{figure=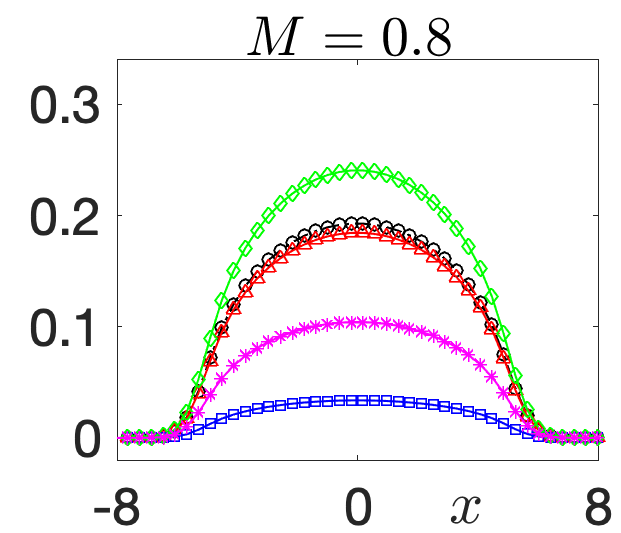,width=4.5cm,height=3.2cm}
}
\vspace{0.25cm}
\centerline{
\psfig{figure=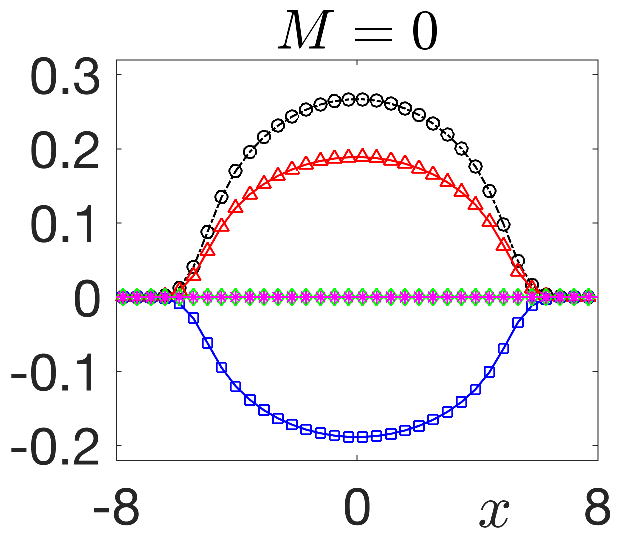,width=4.5cm,height=3.2cm}\qquad\qquad
\psfig{figure=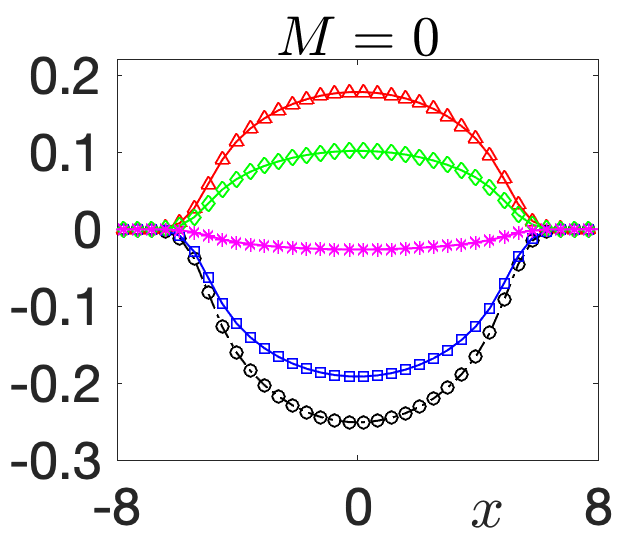,width=4.5cm,height=3.2cm}
}
\vspace{0.25cm}
\centerline{
\psfig{figure=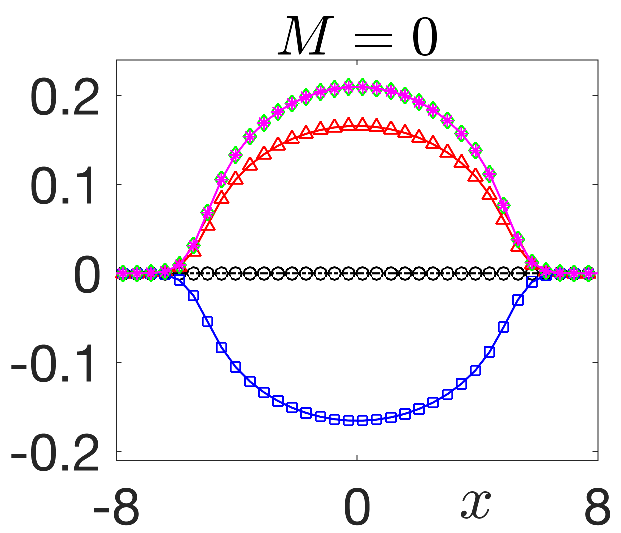,width=4.5cm,height=3.2cm}\qquad\qquad
\psfig{figure=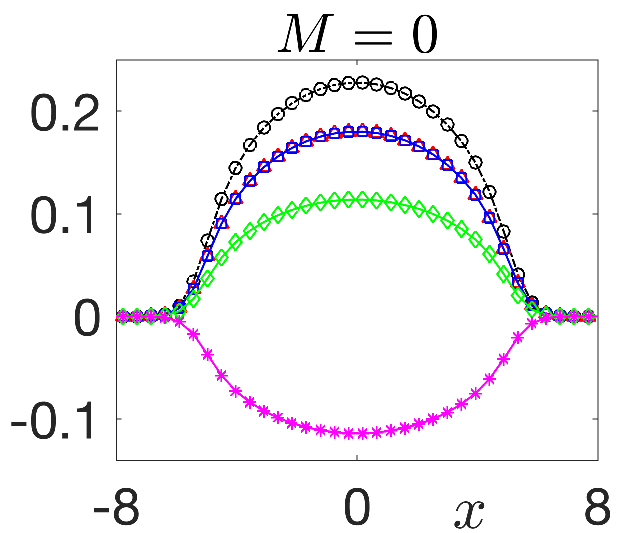,width=4.5cm,height=3.2cm}
}
 \caption{Plots of the wave function of the ground states
 $\phi_\ell$ ($\ell=2, 1, 0, -1, -2$) in  Cases 1-4  (from top to bottom)  that computed by different 
 initial data in  Example \ref{eg:num_1D_nonuniqueness}.}
 \label{fig: diff_GS}
\end{figure}

\subsection{Applications}
In this subsection, we apply our numerical method to compute the ground state of spin-2 BEC
with different parameter regimes.  

\medskip

 \begin{exmp}
 \label{eg:num_1D_nonuniqueness}
 Here we further study the non-uniqueness of the ground state under some parameter regimes. We take $d=1$, and carry out the following cases:
\begin{itemize}
 \item[(i)] Case 1.  $\beta_1=-1$, $\beta_2=-20$, $M=0.5$,  $\bm{\xi}_g$  is taken as  \eqref{GSCase2Type1}  and \eqref{GSCase1Type2}, respectively; 

 \item[(ii)] Case 2.  $\beta_1=0$, $\beta_2=1$, $M=0.8$, $\bm{\xi}_g$ is taken as  \eqref{GSCase2Type1} and \eqref{h414}, respectively;
 
 \item[(iii)] Case 3.  $\beta_1=10$, $\beta_2=2$, $M=0$,  $\bm{\xi}_g$  is taken as  \eqref{GSCase3Type311}  and  \eqref{GSCase3Type3}    with  $\theta=\arcsin(1/5)$, respectively;
 
 \item[(iv)] Case 4.  $\beta_1=1$, $\beta_2=-2$, $M=0$,  $\bm{\xi}_g$ is taken as 
  \eqref{GSCase1Type11} with $(\gamma_1=0.4,       \theta=\arcsin\sqrt{2/15})$  and  
       $\theta=\arcsin(2\sqrt{2}/5)$, respectively.         
\end{itemize}
 \end{exmp}
 
Fig. \ref{fig: diff_GS} depicts the wave functions of different ground states computed by different initial data in Cases 1-4. For each case, different initial data converges to different ground states with the same energy, {\color{black} which are $\mathcal{E}(\Phi^g)=8.28198899$, $8.50852656$, $8.50852656$ and $8.48600868$ for Cases 1-4, respectively}.  From Fig. \ref{fig: diff_GS} and additional results not shown here for brevity, we can see that the ground states  are not unique for the following four cases (cf. Fig. \ref{fig: diff_GS}): 
(a).  $\forall\, M\in[0, 2)$ and $\beta_1<0$, $\beta_2=20\beta_1$.
(b).  $\forall\, M\in[0, 2)$ and $\beta_2>0$, $\beta_1=0$.
(c).  $M=0$ and  $\beta_1>0$, $\beta_2>0$. 
(d).  $M=0$  and $\beta_2<0$, $\beta_2<20\beta_1$.

\medskip

 \begin{exmp}
 \label{eg:num_1D_gs_etc}
 In order to study the wave functions and SMA property of the ground states
 in different parameter regimes. We take $d=1$, initial data as \eqref{ini_choice} 
with $\bm{\xi}_g$ reading as \eqref{GSCase2Type1}, \eqref{GSCase1Type2} 
and \eqref{h414}, and consider following three cases:
\begin{itemize}
 \item[(i)]  Case 5. ferromagnetic phase,  we take $\beta_1=-1$ and $\beta_2=2$;

 \item[(ii)] Case 6. nematic phase, we choose $\beta_1=1$ and $\beta_2=-2$; and

 \item[(iii)] Case 7. cyclic phase, we let  $\beta_1=10$ and $\beta_2=2$.
\end{itemize}

%

 \end{exmp}

 \begin{figure}[htbp!]
\centerline{
\psfig{figure=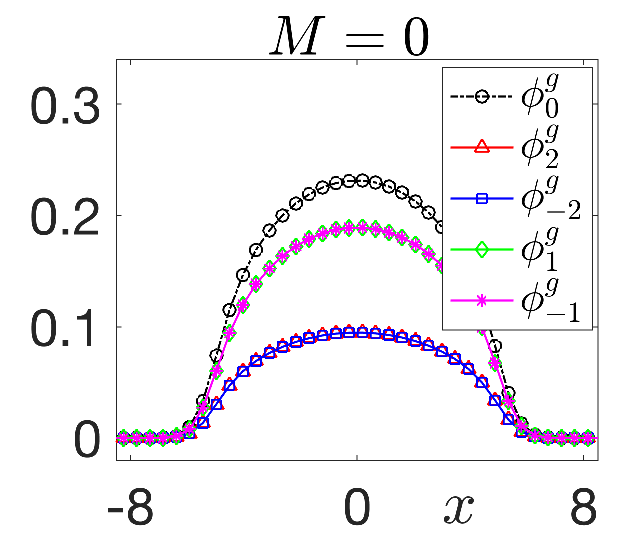,width=3.8cm,height=3.cm}
\hspace{0.3cm}
\psfig{figure=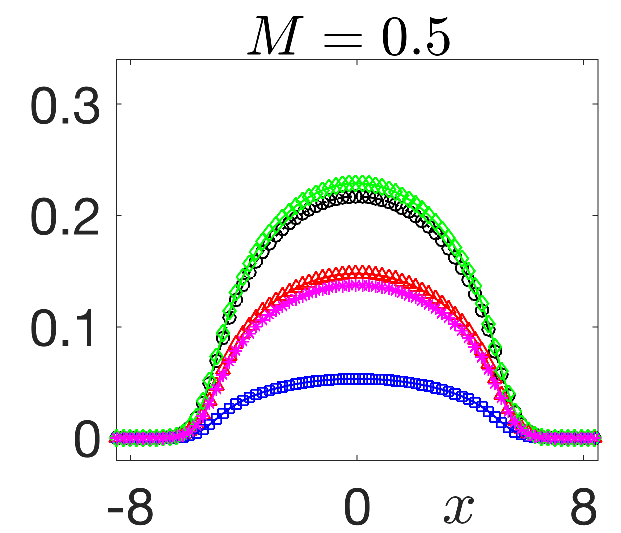,width=3.8cm,height=3.cm}
\hspace{0.3cm}
\psfig{figure=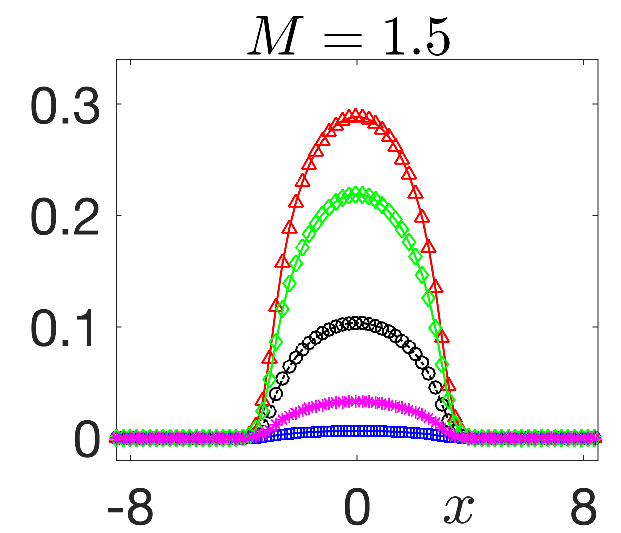,width=3.8cm,height=3.cm}
}
\vspace{0.3cm}
\centerline{
\psfig{figure=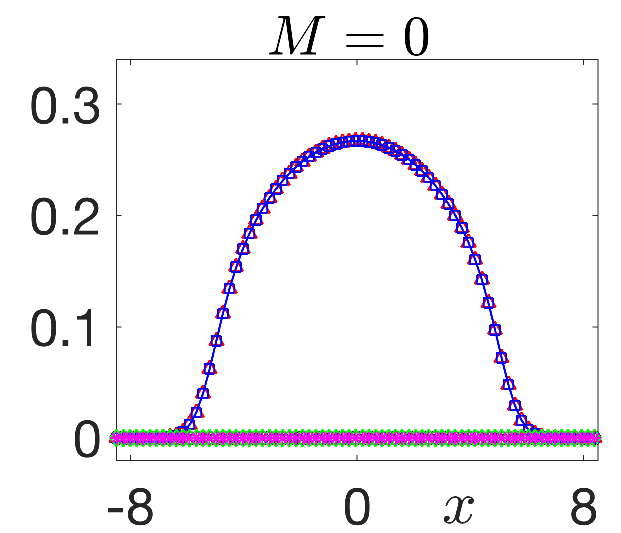,width=3.8cm,height=3.cm}
\hspace{0.3cm}
\psfig{figure=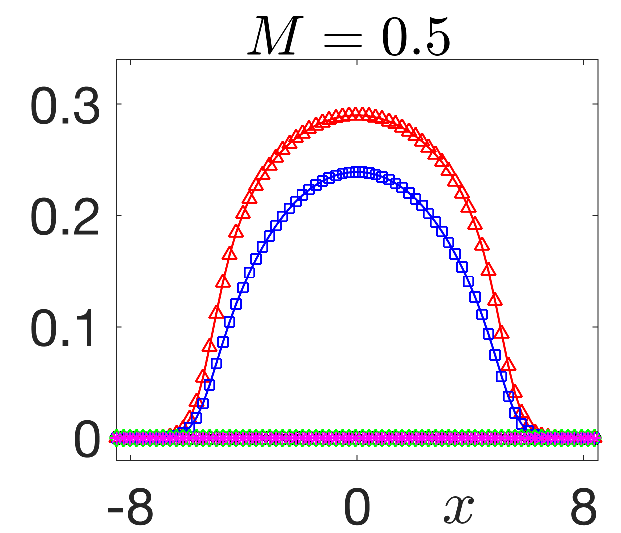,width=3.8cm,height=3.cm}
\hspace{0.3cm}
\psfig{figure=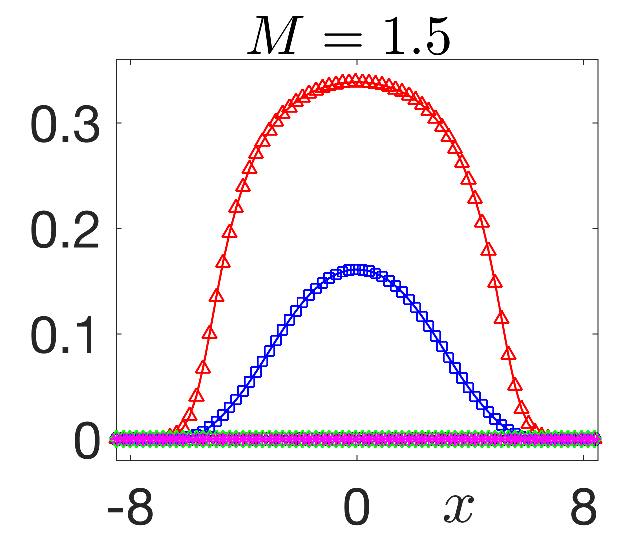,width=3.8cm,height=3.cm}
}
\vspace{0.3cm}
\centerline{
\psfig{figure=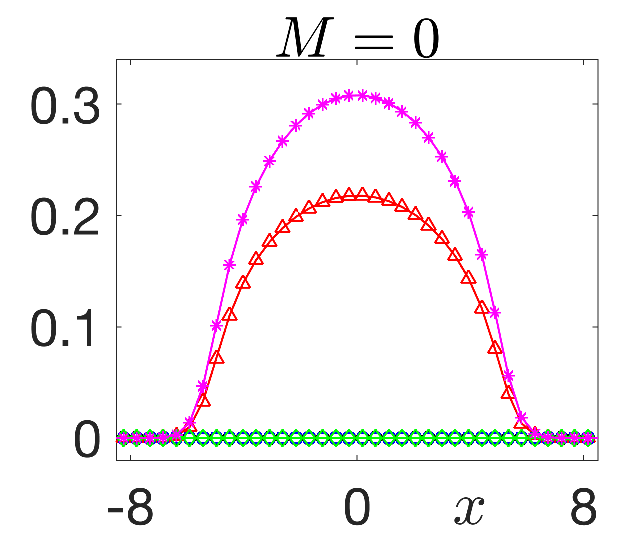,width=3.8cm,height=3.cm}
\hspace{0.3cm}
\psfig{figure=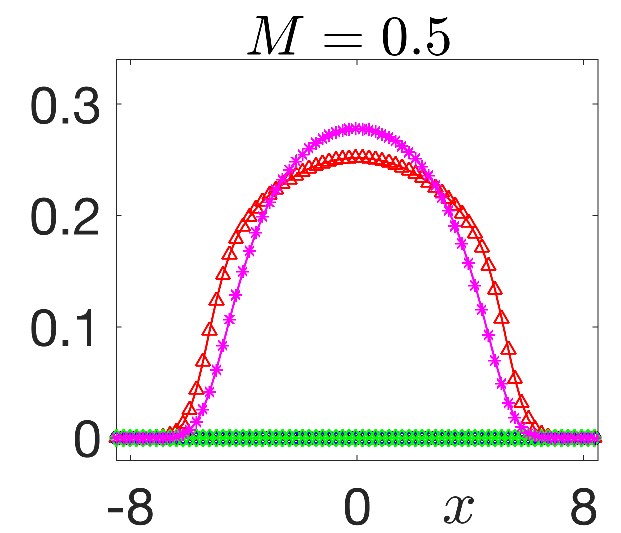,width=3.8cm,height=3.cm}
\hspace{0.3cm}
\psfig{figure=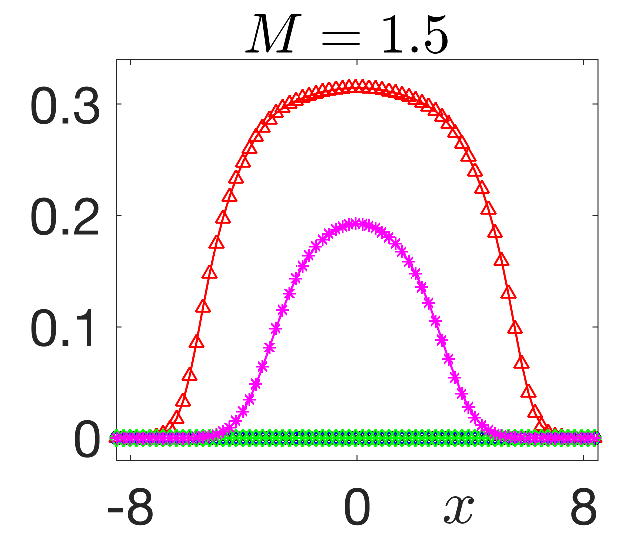,width=3.8cm,height=3.cm}
}
\caption{Wave functions of ground states, i.e., $\phi_\ell^g$ ($\ell=2,1,0,-1,-2$)  with 
different magnetizations $M=0,\; 0.5,\; 1.5$ (left to right)  for Cases 5-7 
(top to bottom)  in  Example \ref{eg:num_1D_gs_etc}.}
\label{fig: gs_1d_case5to7}
\end{figure}

 \begin{table}[!htbp]
\centering\small
{\color{black}
\caption{The component masses $\mathcal{N}_\ell$ ($\ell=2,1,0,-1,-2$), total masses $\mathcal{N}(\Phi^g)$ and total energies $\mathcal{E}(\Phi^g)$ of the ground states $\Phi^g$ for Cases 5-7 (top to bottom) in Example \ref{eg:num_1D_gs_etc}.}
\label{tab:mass_eng}
\begin{tabular}{ccccccccc}
\hline
$(\beta_1,\beta_2)$ & $M$ & $\mathcal{N}_2$ & $\mathcal{N}_1$ & $\mathcal{N}_0$ & $\mathcal{N}_{-1}$ & $\mathcal{N}_{-2}$ & $\mathcal{N}(\Phi^g)$ & $\mathcal{E}(\Phi^g)$\\
\hline
\multirow{3}{*}{$(-1,2)$}
& 0 & 0.0627 & 0.2500 &0.3744 & 0.2500& 0.0627 &1.0000 & 8.2820 \\
& 0.5 &0.1530 &0.3659 & 0.3290 & 0.1321&0.0199 &1.0000 & 8.2820 \\
& 1.5 & 0.5865 & 0.3343 & 0.0720& 0.0069 &0.0002 &1.0000 & 8.2820\\
\hline
\multirow{3}{*}{$(1,-2)$}
& 0 & 0.5000& 0 & 0 & 0& 0.5000 &1.0000 & 8.4860\\
& 0.5 & 0.6250 & 0 & 0 &0&  0.3750 &1.0000  & 8.5003\\
& 1.5 & 0.8750 & 0 &0& 0&0.1250&1.0000  &8.6187\\
\hline
\multirow{3}{*}{$(10,2)$}
& 0 & 0.3333 & 0 & 0 & 0.6667 & 0 &1.0000 &8.5085 \\
& 0.5 &0.5000 & 0& 0 & 0.5000 & 0 &1.0000 & 8.6309 \\
& 1.5 &0.8333 & 0 &0  & 0.1667 &0 &1.0000  & 9.6496 \\
\hline
\end{tabular}
}
\end{table}

\begin{figure}[htbp!]
\centerline{
\psfig{figure=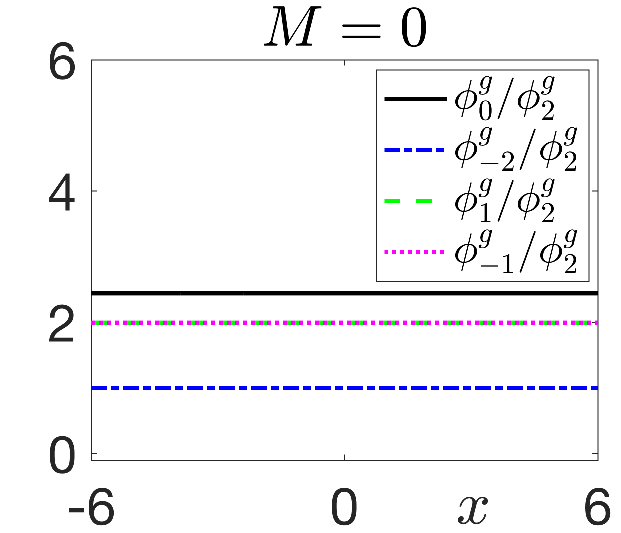,width=3.8cm,height=3.cm}
\hspace{0.3cm}
\psfig{figure=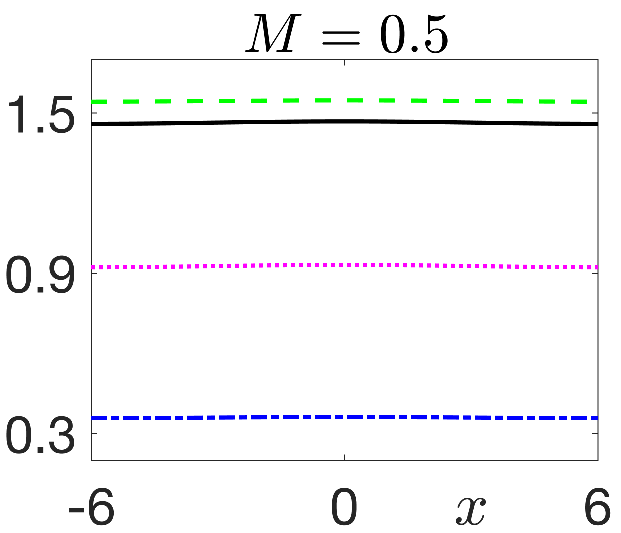,width=3.8cm,height=3.cm}
\hspace{0.3cm}
\psfig{figure=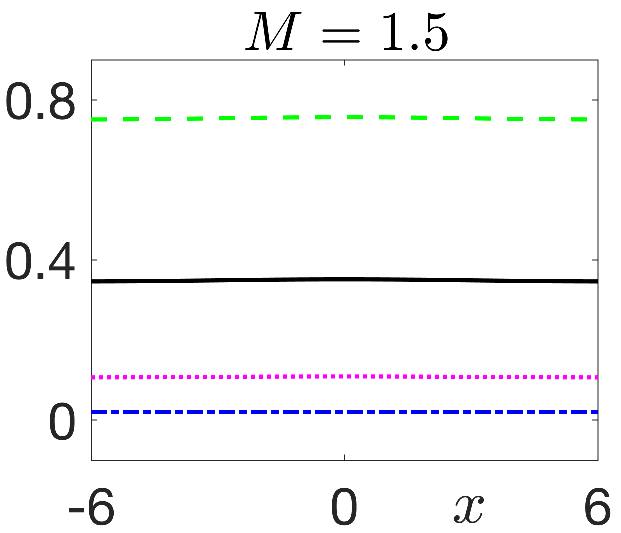,width=3.8cm,height=3.cm}
}
\vspace{0.3cm}
\centerline{
\psfig{figure=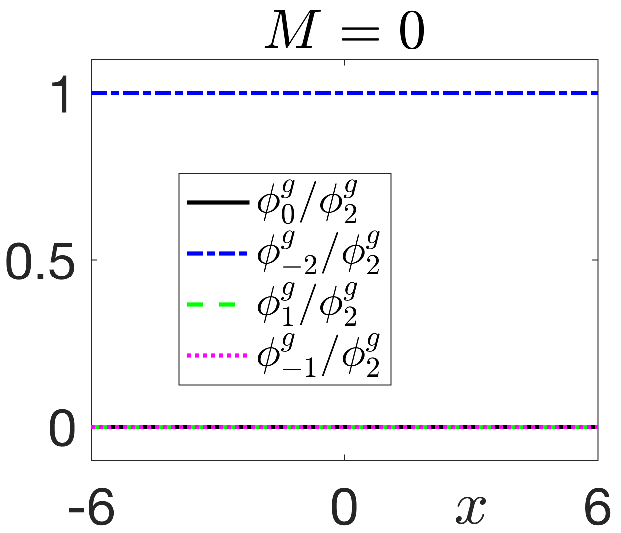,width=3.8cm,height=3.cm}
\hspace{0.3cm}
\psfig{figure=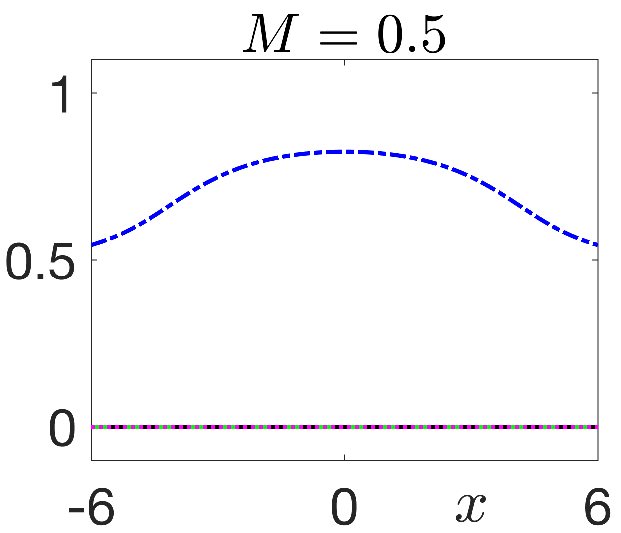,width=3.8cm,height=3.cm}
\hspace{0.3cm}
\psfig{figure=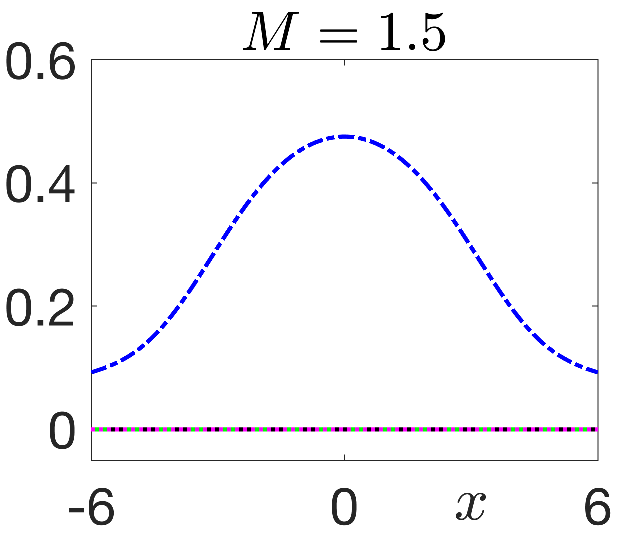,width=3.8cm,height=3.cm}
}
\vspace{0.3cm}
\centerline{
\psfig{figure=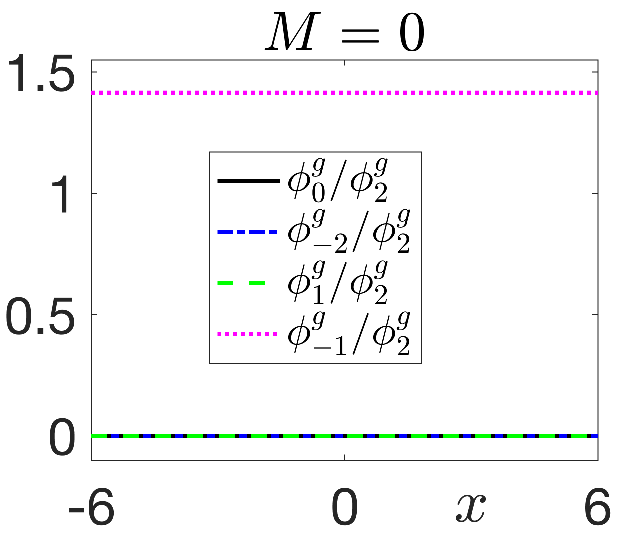,width=3.8cm,height=3.cm}
\hspace{0.3cm}
\psfig{figure=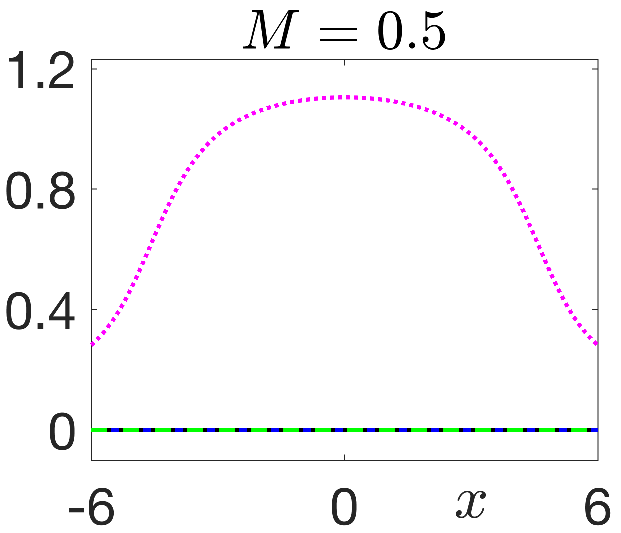,width=3.8cm,height=3.cm}
\hspace{0.3cm}
\psfig{figure=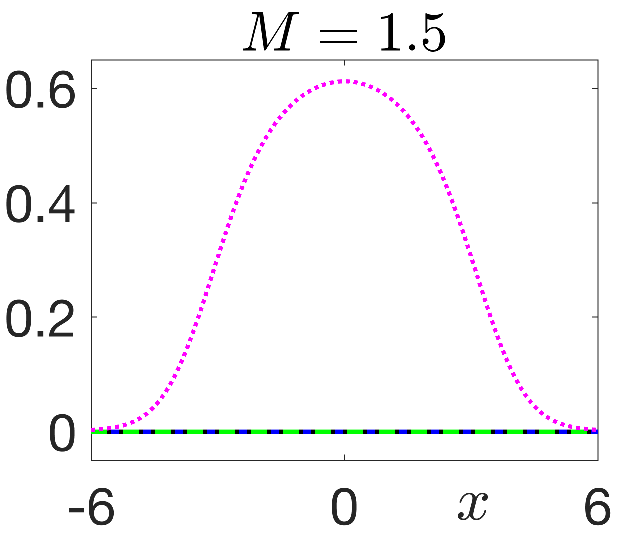,width=3.8cm,height=3.cm}
}
\caption{Plots of $\phi^g_1/\phi^g_2$ (dashed line), $\phi^g_0/\phi^g_2$  (solid line),  $\phi^g_{-1}/\phi^g_2$  (dotted line), and $\phi^g_{-2}/\phi^g_2$  (dashed-dotted line) for Cases 5-7 (top to bottom) in  Example \ref{eg:num_1D_gs_etc}, respectively, to analyze the SMA property for different parameters $(\beta_1,\beta_2,M)$. }
\label{fig_sma} 
\end{figure}

\smallskip
 
Fig. \ref{fig: gs_1d_case5to7} depicts the wave functions of the ground states in Cases 5-7 for magnetizations $M=0$, $0.5$ and $1.5$, respectively, {\color{black}while Table \ref{tab:mass_eng} shows the component masses $\mathcal{N}_\ell$ ($\ell=2,1,0,-1,-2$), total masses $\mathcal{N}(\Phi^g)$ and total energies $\mathcal{E}(\Phi^g)$ of the corresponding ground states.} Fig. \ref{fig_sma} shows the SMA property for different  $(\beta_1,\beta_2,M)$.\par


From Figs. \ref{fig: gs_1d_case5to7}-\ref{fig_sma}, Table \ref{tab:mass_eng} and extensive numerical experiments not shown here for brevity, we observe that: (i). When $M\in(0, 2)$, for ferromagnetic phase, $\phi_\ell^g>0$ for all $\ell=-2,-1,0,1,2$ (cf. Fig. \ref{fig: gs_1d_case5to7} (top row)). For nematic phase, $\phi_2^g>0$ \& $\phi_{-2}^g>0$, while $\phi_1^g=\phi_{-1}^g=\phi_0^g\equiv0$ (cf. Fig. \ref{fig: gs_1d_case5to7} (middle row)). For  cyclic phase, $\phi_2^g>0$ \& $\phi_{-1}^g>0$,  while $\phi_1^g=\phi_{-2}^g=\phi_0^g\equiv0$  (cf. Fig. \ref{fig: gs_1d_case5to7}( bottom row)). {\color{black}(ii). The component masses $\mathcal{N}_\ell$ ($\ell=2,1,0,-1,-2$) of the ground state of non-uniform spin-2 BEC system are the same as those of the corresponding spatial uniform system. Meanwhile, the total energy of the ferromagnetic ground states are independent of the magnetization $M$, whereas the total energy of the nematic or cyclic ground states are increased with the magnetization $M$ (cf. Table \ref{tab:mass_eng}).} (iii). The SMA is valid for the ground states with  ferromagnetic phase and those with the nematic or cyclic phases as well as zero magnetization, however, it is invalid for the other cases (cf. Fig.~\ref{fig_sma}). 

 \begin{exmp}
 \label{eg:num_2D}
 Here, we study the ground state of a two-dimensional
  spin-2 BEC  with harmonic/box/ optical lattice potentials. To this end,
 we take $d=2$, $q_1=q_2=\pi/2$, $\eta=0$/$\eta=10$ in \eqref{num_choice} for the harmonic/optical lattice potential, and choose a box potential $V_{box}(\bx)$ as in \eqref{box_potential}. We consider the following three cases: 
\begin{itemize}
 \item[(i)] Case 9. ferromagnetic phase,  let $\beta_1=-1$ {\rm \&} $\beta_2=-5$;
 
 \item[(ii)] Case 10.   nematic phase,  choose $\beta_1=-1$ {\rm \&} $\beta_2=-25$;
 
 \item[(iii)] Case 11.  cyclic phase,  take $\beta_1=10$ {\rm \&} $\beta_2=2$.
 \end{itemize}
 \end{exmp}

 \begin{figure}[htbp!]
\hspace{-1.5cm}
\psfig{figure=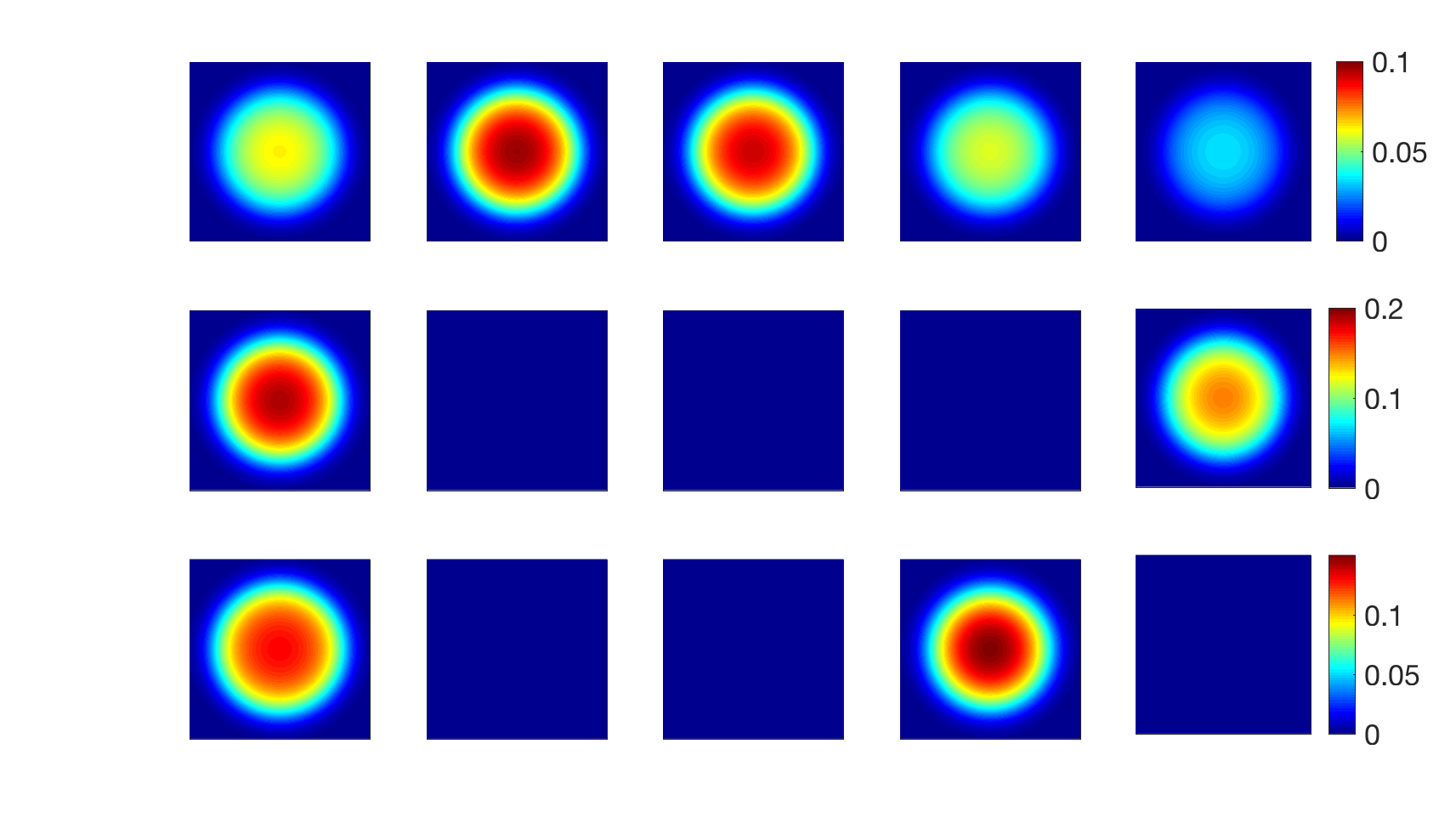,width=5.8in,height=3.0in}
\vspace{-1.2cm}
 \caption{$M=0.5$.  Contour plots of the components of the ground states  $\phi_\ell^g$ 
 (from left to right, $\ell=2,1,0,-1,-2$)  in  Cases 9-11 (top to bottom) with the harmonic potential in  Example \ref{eg:num_2D}.
 }
\label{fig16}
\end{figure}

\begin{figure}[htbp!]
\hspace{-1.5cm}
\psfig{figure=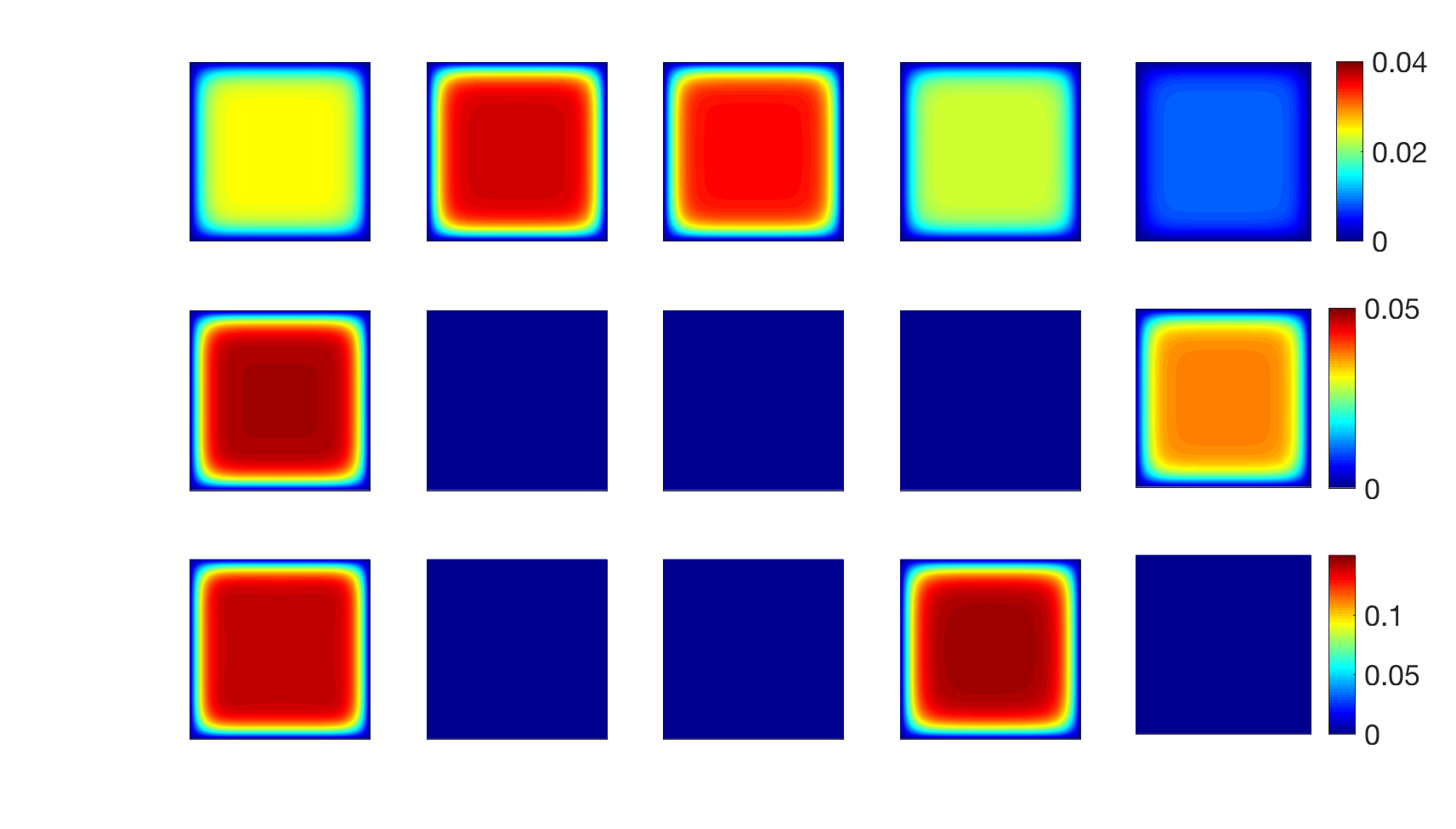,width=5.8in,height=3.2in}
\vspace{-1.2cm}
 \caption{$M=0.5$.  Contour plots of the components of the ground states  $\phi_\ell^g$ 
 (from left to right, $\ell=2,1,0,-1,-2$)  in Cases 9-11 (top to bottom) with the box potential in  Example \ref{eg:num_2D}.
 }
\label{fig17}
\end{figure}

\begin{figure}[htbp!]
\hspace{-1.5cm}
\psfig{figure=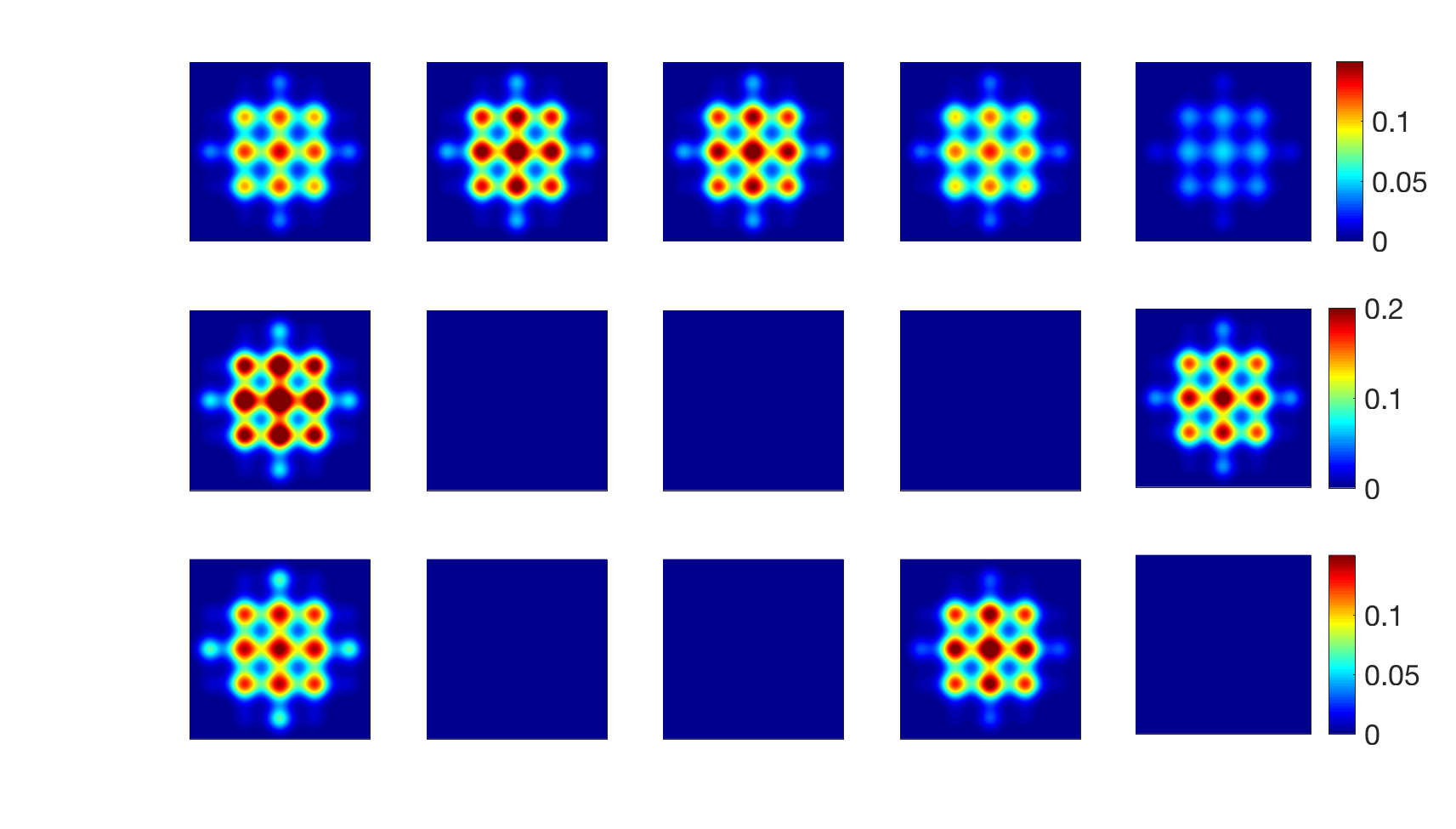,width=5.8in,height=3.2in}
\vspace{-1.2cm}
 \caption{$M=0.5$.  Contour plots of the components of the ground states  $\phi_\ell^g$ 
 (from left to right, $\ell=2,1,0,-1,-2$)  in  Cases 9-11  (top to bottom) with the optical lattice potential in  Example \ref{eg:num_2D}.
 }
\label{fig18}
\end{figure}

 \begin{table}[!htbp]
\centering\small
{\color{black}
\caption{The component masses $\mathcal{N}_\ell$ ($\ell=2,1,0,-1,-2$) and total energies $\mathcal{E}(\Phi^g)$ of the ground states $\Phi^g$ for Cases 9-11 (top to bottom) in Example \ref{eg:num_2D}.}
\label{tab:mass_eng2D}
\begin{tabular}{cccccccc}
\hline
$V(x,y)$ &$(\beta_1,\beta_2)$ & $\mathcal{N}_2$ & $\mathcal{N}_1$ & $\mathcal{N}_0$ & $\mathcal{N}_{-1}$ & $\mathcal{N}_{-2}$  & $\mathcal{E}(\Phi^g)$\\
\hline
\multirow{3}{*}{$\fl{1}{2}\sum\limits_{j=1}^2\nu^2_j$}
& $(-1,-5)$& 0.1526 & 0.3662 &0.3296 & 0.1318& 0.0198  & 3.8727 \\
& $(-1,-25)$ &0.6250 &0 & 0 & 0&0.3750  & 3.8553 \\
& $(10,2)$ & 0.5000 & 0 & 0& 0.5000 &0  & 3.9848\\
\hline
\multirow{3}{*}{$\fl{\sum\limits_{j=1}^2\big[\nu^2_j+20\sin^2\big(\frac{\pi\nu_j}{2}\big)\big]}{2}$}
& $(-1,-5)$ & 0.1526 & 0.3662 &0.3296 & 0.1318& 0.0198  &  11.7247\\
& $(-1,-25)$ & 0.6250 & 0 & 0 &0&  0.3750   & 11.7012\\
& $(10,2)$ & 0.5000 & 0 &0& 0.5000&0  &11.8746\\
\hline
\multirow{3}{*}{$V_{\rm box}(x,y)$}
& $(-1,-5)$ & 0.1526 & 0.3662 &0.3296 & 0.1318& 0.0198  & 0.2024 \\
& $(-1,-25)$ & 0.6250 & 0 & 0 &0&  0.3750   & 0.2009 \\
& $(10,2)$ & 0.5000 & 0 &0& 0.5000&0  & 0.2128 \\
\hline
\end{tabular}
}
\end{table}

Figs.  \ref{fig16}-\ref{fig18} show the plots of the wave functions of the ground states 
in Cases 9-11 with the harmonic, box and optical lattice potential, respectively, {\color{black}and Table \ref{tab:mass_eng2D} presents the component masses $\mathcal{N}_\ell$ ($\ell=2,1,0,-1,-2$) and total energies $\mathcal{E}(\Phi^g)$ of the corresponding ground states.} From these results and additional numerical experiments not shown here for brevity, one finds that our method can be applied to compute the ground state of spin-2 BEC with general potentials. {\color{black} The component masses $\mathcal{N}_\ell$ are independent of the types of potentials, but the energies of the ground states are changed with different types of potentials (cf. Table \ref{tab:mass_eng2D})}. Additionally,  similar as the 1-d case,  the uniqueness,  validity of SMA and phenomena of vanishing-component of the ground state can also be concluded in the two-dimensional case.

\section{Conclusion}
We proposed an efficient and accurate normalized gradient flow method for computing 
the ground states of spin-2 BEC by introducing three additional projection constraints, in addition to 
the conservation of the total mass and magnetization.  A backward-forward finite difference method
was applied to fully discretize the gradient flow with discrete normalization. 
Moreover,  the ground states in  spatially uniform system, i.e $V(\bx)=0$, were solved analytically, which 
give important hints  to understand the properties of  ground states in spatially  non-uniform systems
as well as to the choice of initial data for numerical calculations.
The numerical method was then applied to study  the ground states of 
spin-2 BECs with {\color{black}  ferromagnetic, nematic and cyclic 
phases under harmonic, box and optical lattice  potentials in  both one- and two-dimension.} 
Various numerical experiments  were carried out, which suggest some    
interesting properties about the ground states. For example,  
the parameter regimes for the uniqueness,  validity of SMA and phenomena of vanishing-component were  numerically partially found. Rigorous mathematical justifications for these observations are on-going.


\end{document}